\documentclass[a4paper,12pt,leqno]{article}

\usepackage{amssymb,amsmath}
\usepackage[abbrev]{amsrefs}
\usepackage{amsthm}
\usepackage{hyperref}
\usepackage{xy}
\usepackage{color}

\xyoption{all}

\numberwithin{equation}{section}

\newtheorem{defn}{Definition}[section]
\newtheorem{corollary}[defn]{Corollary}
\newtheorem{rem}[defn]{Remark}
\newtheorem{exm}[defn]{Example}
\newtheorem{lemma}[defn]{Lemma}
\newtheorem{theorem}[defn]{Theorem}

\newenvironment{remark}{\begin{rem}\em}{\end{rem}}
\newenvironment{example}{\begin{exm}\em}{\end{exm}}

\newcommand\V{\bigvee}
\newcommand\ie{i.e.}
\newcommand\st{\mid}
\newcommand\cf{\textrm{cf.}}
\newcommand\opens{\operatorname{\mathcal{O}}}
\newcommand\topology{\operatorname{\Omega}}
\newcommand\spp{\varsigma}
\newcommand\downsegment{{\downarrow}}
\newcommand\Frm{\textit{Frm}}
\newcommand\Loc{\textit{Loc}}
\newcommand\Top{\textit{Top}}
\newcommand\opp[1]{{#1}^{\textrm{op}}}
\newcommand\ident{\mathrm{id}}
\newcommand\ipi{\mathcal I}
\newcommand\SL{\mathit{SL}}
\newcommand\lcc{\operatorname{{\mathcal L}^{\vee}}}
\newcommand\biIQLoc{\mathbf{IQLoc}}
\newcommand\grpd{\mathit{Gpd}}
\newcommand\HSG{\mathit{Gpd}_{\mathrm{HS}}}
\newcommand\HSGsh{\mathit{Gpd}_{\mathrm{HS}}^{\mathrm{sh}}}
\newcommand\HSQ{\mathit{InvQuF}_{\mathrm{HS}}}
\newcommand\bigrpd{\mathbf{Gpd}}
\newcommand\sections{\mathit\Gamma}
\newcommand\psections{\sections^{\mathrm{pr}}}
\newcommand\bisections{\sections^{\mathrm{bi}}}
\newcommand\tsections{\sections^{\tspp}}
\newcommand\tspp{\tau}
\newcommand\act{\mathfrak a}
\newcommand\bct{\mathfrak b}
\newcommand\Sh{\mathit{Sh}}

\begin{document}

\title{Functoriality of groupoid quantales. II\thanks{Work funded by FCT/Portugal through the LisMath program and project PEst-OE/EEI/LA0009/2013.}}

\author{Juan Pablo Quijano and Pedro Resende}

\date{~}

\maketitle

\vspace*{-1cm}
\begin{abstract}
Taking advantage of the quantale-theoretic description of \'etale groupoids we study principal bundles, Hilsum--Skandalis maps, and Morita equivalence in terms of modules on inverse quantal frames. The Hilbert module description of quantale sheaves leads naturally to a formulation of Morita equivalence in terms of bimodules that resemble imprimitivity bimodules of C*-algebras.
\\
\vspace*{-2mm}~
\vspace*{-2mm}~
\end{abstract}

{\small{\tableofcontents}}

\section{Introduction}\label{introduction}

Inverse quantal frames are a ``ring theoretic'' description of \'etale groupoids \cite{Re07}. This correspondence extends naturally to a module theoretic description of actions and sheaves for \'etale groupoids \cite{GSQS} whereby, for instance, sheaves can be identified with Hilbert quantale modules \cite{Paseka} equipped with a ``Hilbert basis''. The latter correspondence is functorially well behaved, but the correspondence between \'etale groupoids and inverse quantal frames themselves is not functorial with respect to the usual notions of morphism (groupoid functors and homomomorphisms of unital involutive quantales) unless the classes of morphisms are seriously restricted. Instead, it is functorial in a bicategorical sense \cite{Re15}: the bicategory $\bigrpd$, of localic \'etale groupoids with bi-actions as 1-cells, is bi-equivalent to the bicategory $\biIQLoc$, whose objects are the inverse quantal frames and whose 1-cells are quantale bimodules that satisfy a mild condition.

A bi-action between \'etale groupoids $G$ and $H$ is a span of locales $G_0\leftarrow X\rightarrow H_0$ together with a left $G$-action and a right $H$-action on $X$ satisfying natural conditions. It can be thought of as a binary relation between the generalized spaces (in the topos sense) of orbits of the groupoids. The corresponding notion of functional relation is that of a Hilsum--Skandalis map~ \cites{Moer90,Moer87,Bunge,MRW87,La01,Mr99,HS87}. Concretely, this is defined to be the isomorphism class of a \emph{principal bibundle}, by which is meant a bi-action whose right projection $X\rightarrow H_0$ is an open surjection (and necessarily a local homeomorphism) and whose left action makes the right projection a principal $G$-bundle. Principal bibundles whose right projection has a global section correspond precisely to groupoid functors. Based on Hilsum--Skandalis maps one may also give one of several equivalent definitions of Morita equivalence for \'etale groupoids.

The purpose of this paper is to describe Hilsum--Skandalis maps and Morita equivalence of \'etale groupoids in quantale language in order to provide a better algebraic understanding of these concepts, in particular one that yields a Morita theory which is closer to that of rings. This should also be a useful asset when studying Morita equivalence of other structures that relate closely to inverse quantal frames, such as pseudogroups (complete infinitely distributive inverse semigroups), whose category is equivalent to that of inverse quantal frames but for which a specific notion of Morita equivalence is currently lacking. We briefly return to this point at the end of the paper.

Our principal bundles and Hilsum--Skandalis maps will be defined to be quantale modules, in fact sheaves (in the sense of~\cite{GSQS}) because the ``equivalence bimodules'' (biprincipal bibundles) for \'etale groupoids are necessarily sheaves. Then the description of sheaves by Hilbert modules has the natural consequence that our notions of principal bundle, etc., are quantale modules equipped with quantale-valued inner products. This leads to a notion of equivalence bimodule that resembles the imprimitivity bimodules of strong Morita equivalence of C*-algebras, and is also close to the equivalence bimodules of inverse semigroups of~\cite{Steinberg-morita}.

Although in principle it would be interesting to know whether our definitions can be adapted to more general quantales than inverse quantal frames, we shall not attempt to find this here because it would add unnecessary technical complications in an already long paper and also because we do not currently have specific examples in mind of quantales at the required level of generality. For instance, stably supported quantales are interesting in their own right~\cite{MarcR}, but it is likely that they are already too general for the theory presented in this paper to work. It should be noted that the theories of Morita equivalence of~\cite{BorVit} and~\cite{Paseka3} apply to even more general quantales but do not at all coincide with ours when restricted to inverse quantal frames.

Our main results are in section~\ref{sec:principalsheaves}, preceded by section~\ref{sec:principalbundles} where the topological description of Hilsum--Skandalis maps and Morita equivalence is adapted to the setting of localic groupoids. In addition to the main results, we need and prove, in section~\ref{sec:oms}, new results of independent interest concerning sheaves on locales and inverse quantal frames, notably Theorem~\ref{innerproduct2} which provides a useful formula for computing the inner product of a complete Hilbert module.

\section{Preliminaries}\label{sec:prelim}

This section is mostly for fixing terminology and notation. We recall some basic facts concerning the relation between \'etale groupoids and quantales, on one hand, and the relation between (bi-)actions of \'etale groupoids and quantale (bi-)modules, respectively taken from \cite{Re07} and \cites{GSQS,Re15}.

\subsection{Groupoid quantales}

Let us begin by recalling basic notions and notation concerning sup-lattices, locales, groupoids and their quantales, mostly following~\cites{JT,stonespaces,Re07,GSQS,Re15}.

\paragraph{Sup-lattices and locales.}

By a \emph{sup-lattice} is meant a complete lattice, and a \emph{sup-lattice homomorphism} $h:Y\to X$ is a mapping that preserves arbitrary joins. The resulting \emph{category of sup-lattices} $\SL$ is bi-complete and monoidal \cite{JT}. The top element of a sup-lattice $X$ is denoted by $1_X$ or simply $1$, and the bottom element by $0_X$ or simply $0$. Given $x\in X$ we use the following notation: $\downsegment(x):=\{y\in X\st y\le x\}$.
A sup-lattice further satisfying the infinite distributive law
\[
x\wedge\V_{i} y_i = \V_{i} x\wedge y_i
\]
is a \emph{frame}, or \emph{locale},
and a \emph{frame (or locale) homomorphism} $h:Y\to X$ is a sup-lattice homomorphism that preserves finite meets. This defines the \emph{category of frames}, $\Frm$.
The dual category $\Loc=\opp\Frm$ is referred to as the \emph{category of locales} \cite{stonespaces}, and its arrows are called \emph{continuous maps}, or simply \emph{maps}. These categories are bi-complete, and the product of $X$ and $Y$ in $\Loc$ is denoted by $X\otimes Y$, since it coincides with the tensor product in $\SL$ \cite[\S I.5]{JT}.

If $f:X\to B$ is a map of locales we refer to the corresponding frame homomorphism $f^*:B\to X$ as its \emph{inverse image}. Such a homomorphism turns $X$ into a \emph{$B$-module} (in the sense of quantale modules --- see section~\ref{sec:prelimactions}) with action
\[
b x=f^*(b)\wedge x
\]
for all $b\in B$ and $x\in X$, and the map $f$ is \emph{semiopen} if $f^*$ has a left adjoint $f_!:X\to B$, referred to as the \emph{direct image} of $f$. If $f_!$ is a homomorphism of $B$-modules then $f$ is \emph{open}. The $B$-equivariance of $f_!$ is known as the \emph{Frobenius reciprocity condition}. A \emph{local homeomorphism} $f:X\to B$ is a (necessarily open) map for which there is a subset $\sections\subset X$ satisfying $\V \sections =1$ (a \emph{cover} of $X$) such that for each $s\in \sections$ the direct image $f_!$ restricts to an isomorphism $\downsegment (s)\cong\downsegment (f_!(s))$.
Both open maps and local homeomorphisms are stable under pullbacks.

Let $B$ be a locale. By a \emph{$B$-locale} will be meant a locale $X$ equipped with a structure of (left) $B$-module satisfying the condition
\begin{equation}
bx=b1_X \wedge x
\end{equation}  
for all $b\in B$ and $x\in X$. The $B$-module obtained from a map $p:X\to B$ as described above is a $B$-locale. Conversely, from a $B$-locale $X$ a map $p:X\to B$ is obtained by defining $p^*(b)=b1_X$ for all $b\in B$. We refer to such a map as the \emph{projection}, or \emph{anchor map}, of the $B$-locale.

The \emph{category of $B$-locales} $B$-$\Loc$ consists of the $B$-locales as objects, with arrows being the maps $f:X\to Y$ of locales whose inverse image $f^*:Y\to X$ is a homomorphism of $B$-modules, and it is isomorphic to the slice category $\Loc/B$.

Let $B$ be a locale and let $X$ and $Y$ be $B$-locales. Their product in $B$-$\Loc$ is the pullback of the anchor maps, and it coincides with the tensor product $X\otimes_B Y$ of $B$-modules.

\paragraph{\'Etale groupoids.}

A \emph{localic groupoid} (resp.\ \emph{topological groupoid}) is an internal groupoid in the category of locales $\Loc$ (resp.\ $\Top$). We denote the locales (resp.\ spaces) of objects and arrows of a groupoid $G$ respectively by $G_0$ and $G_1$,
\[
\xymatrix{
G=\quad G_2\ar[r]^-m&G_1\ar@(ru,lu)[]_i\ar@<1.2ex>[rr]^r\ar@<-1.2ex>[rr]_d&&G_0\ar[ll]|u
}
\]
where $G_2:=G_1\otimes_{G_0} G_1$ (resp.\ $G_1\times_{G_0} G_1$) is the pullback of the \emph{domain map} $d$ and the \emph{range map} $r$. A groupoid $G$ is said to be \emph{open} if $d$ is an open map. Since the \emph{multiplication map} $m$ can be regarded as a pullback of $d$ (due to the existence of the \emph{inversion map} $i$), if $G$ is open then the map $m$ is open, and in fact the converse is also true, so an open groupoid is precisely a groupoid with open multiplication map.

An \emph{\'etale groupoid} is
an open groupoid such that $d$ is a local homeomorphism, in which case all
the structure maps are local homeomorphisms and, hence, $G_0$ is isomorphic
to an open sublocale (resp.\ homeomorphic to an open subspace) of $G_1$. Conversely, if both $d$ and the \emph{units map} $u$ are open then $G$ is \'etale \cite{Re07}.

When we write, say, $a\in G_1$ for a localic groupoid $G$ we are referring to an element of $G_1$ regarded as a sup-lattice, whereas if $G$ is topological  we use point-set notation and write $g\in G_1$ to denote an arrow of $G$.

\paragraph{Inverse quantal frames.}

By a \emph{unital involutive quantale} is meant an involutive monoid in the monoidal category $\SL$ of sup-lattices. We shall adopt the following terminology and notation:

\begin{itemize}
\item The product of two elements $a$ and $b$ of a unital involutive quantale $Q$ is denoted by $ab$, the involute of $a$ is denoted by $a^*$, and the multiplicative unit is denoted by $e_Q$, or simply $e$.

\item By a \emph{homomorphism} of unital involutive quantales $h:Q\to R$ is meant a homomorphism of involutive monoids in $\SL$.
\end{itemize}

Let $G$ be a localic \'etale groupoid. There is an associated unital involutive quantale $\opens(G)$ whose sup-lattice structure coincides with $G_1$ and whose multiplication, involution and unit are obtained from direct image homomorphisms of structure maps of $G$, as follows:
\begin{itemize}
\item $ab=m_!(a\otimes b)$ for all $a,b\in\opens(G)$;
\item $a^*=i_!(a)$ for all $a\in\opens(G)$;
\item $e=u_!(1_{G_0})$.
\end{itemize}
For a topological groupoid $G$ the associated unital involutive quantale is the topology $\topology(G_1)$ equipped with pointwise multiplication and inversion of open sets, and with $e=u(G_0)$.

Let $Q$ be a unital involutive quantale. We shall denote the downsegment $\downsegment(e)$, which is a unital involutive subquantale, by $Q_0$. We recall that by a \emph{support} on $Q$ is meant a sup-lattice homomorphism $\spp: Q\to Q_0$ (sometimes denoted by $\spp_Q$ if necessary to remove ambiguities) satisfying the following conditions for all $a\in Q$:
\begin{align*}
\spp(a) &\leq aa^*\;,\\
a &\leq \spp(a)a\;.
\end{align*}
The support is said to be \emph{stable} (or \emph{$Q_0$-equivariant}), if in addition the following equivalent conditions hold:
\[\begin{array}{rcll}
\spp(aa') &=& \spp(a\spp(a'))&\textrm{for all }a,a'\in Q\;,\\
\spp(aa') &\le& \spp(a)&\textrm{for all }a,a'\in Q\;,\\
\spp(ba) & =& b\spp(a)&\textrm{for all }a\in Q\text{ and }b\in Q_0\;.
\end{array}\]
If a support is stable then it is the unique support and it is given by
\[
\spp(a) = a1\wedge e = aa^*\wedge e\;.
\]
In this case the quantale is said to be \emph{stably supported}. Moreover, if $Q$ is stably supported then for all $a\in Q$ and $b\in Q_0$ we have
\begin{align*}
ba & = b1\wedge a\;.
\end{align*}

If $G$ is an \'etale groupoid then $\opens(G)$ has a stable support  given by $u_!\circ d_!:\opens(G)\to \opens(G)_0\cong G_0$.

For any unital involutive quantale $Q$ with a support, stable or not, the following equalities hold for all $a, b\in Q$,
\begin{align*}
\spp(a)1 & = a1\;,\\
\spp(b)   & = b\quad \text{if $b\leq e$}\;,
\end{align*}
and $Q_0$ is a locale with meet given by multiplication. We call it the \emph{base locale} of $Q$.

By a  \emph{stable quantal frame} is meant a stably supported quantale which is also a frame, and an \emph{inverse quantal frame} is a stable quantal frame $Q$ that satisfies 
\begin{align*}
\bigvee Q_{\mathcal{I}} &= 1\;,
\end{align*}
where $Q_{\mathcal{I}}$ is the set of \emph{partial units} of $Q$:
\[
Q_\ipi=\{  s\in Q\st ss^* \vee s^*s \leq e  \}\;.
\]
(We are using the notation $Q_\ipi$ instead of the notation of~\cite{Re07}, which was $\ipi(Q)$.)
A simpler characterization has been given in~\cite{SGQ}: a unital involutive quantal frame $Q$ satisfying $\V Q_\ipi=1$ is an inverse quantal frame if and only if it is stably Gelfand, by which it is meant that it satisfies the condition $aa^*a\le a\Rightarrow aa^*a=a$ for all $a\in Q$ (in fact inverse quantal frames satisfy the stronger condition $a\le aa^*a$ for all $a$).

The inverse quantal frames are precisely, up to isomorphisms, the quantales $\opens(G)$ of \'etale groupoids $G$.

\subsection{Groupoid actions as modules}\label{sec:prelimactions}

Now let us recall from \cites{GSQS,Re15} the relations between groupoid (bi-)actions and quantale (bi-)modules.

\paragraph{Groupoid (bi-)actions.}

Let $G$ be a groupoid. A \emph{(left) $G$-locale} is a triple $(X,p,\act)$ where $X$ is a locale together with a map
\[
p: X\to G_0
\]
(called \emph{anchor map} or \emph{projection}) and a map of locales (called the \emph{action})
\[
\act: G_1\otimes_{G_0} X \to X\;,
\]
where $G_1\otimes_{G_0} X$ is the pullback of $r$ and $p$, satisfying the usual associativity and unitarity conditions, and also the commutativity of the following diagram, which in fact is a pullback:
\begin{equation}\label{pi1pbalongd}
\vcenter{\xymatrix{
 G_1\otimes_{G_0} X  \ar[d]_{\pi_1} \ar[r]^-{\act} & X \ar[d]^{p}  \\
 G_1  \ar[r]_-{d} & G_0 }
}
\end{equation}
The \emph{orbit locale} of the $G$-locale $X$ is the locale $X/G$ obtained as the following coequalizer in $\Loc$:
\begin{equation}\label{orbitlocale}
\xymatrix{
G_1\otimes_{G_0} X\ar@<-1.2ex>[rr]_-{\pi_2}\ar@<1.2ex>[rr]^-{\act}&&X\ar[rr]^-{\pi}&&X/G\;.
}
\end{equation}
Since the action of a $G$-locale is a pullback of the domain map of $G$, it follows that if $G$ is an open groupoid the action is necessarily open, and if $G$ is \'etale the action is a local homeomorphism. We also note that if $G$ is open (resp.\ \'etale) the projection $\pi_2:G_1\otimes_{G_0} X\to X$ is an open map (resp.\ local homeomorphism) because $\pi_2$ is a pullback of $d$.

Let $(X,p_X,\act_X)$ and $(Y,p_Y,\act_Y)$ be $G$-locales. We shall write simply $X$ and $Y$ when no confusion will arise. An \emph{equivariant map} from $X$ to $Y$ is a map $f:X\to Y$ in the slice category $\textbf{Loc}/G_0$ that commutes with the actions.
We shall denote the category of $G$-locales and equivariant maps between them by $G$-$\Loc$. 

\begin{lemma}\label{G-loc,isom}
The categories of left $G$-locales and right $G$-locales are isomorphic.
\end{lemma}

\begin{proof}
A left $G$-locale $(X,p,\act)$  yields a right one by defining $\act'(x,g)=\act(g^{-1},x)$ for all $x\in X$ and $g\in G_1$, and, similarly, any right $G$-locale can be turned into a left one. Moreover, a map is equivariant for left actions if and only if it is equivariant for the corresponding right actions. 
\end{proof}

Let $G$ and $H$ be two \'etale groupoids. A \emph{$G$-$H$-bilocale} is a locale $X$ equipped with a left $G$-locale structure $(p, \mathfrak a)$ and a right $H$-locale structure $(q, \mathfrak b)$ such that $q$ commutes with the action of $G$, $p$ commutes with the action of $H$, and there is associativity involving both actions. Using point-set notation (for topological groupoids $G$ and $H$ acting on a topological space $X$) these conditions are easily expressed as follows, for all $x\in X$ and all arrows $g\in G$ and $h\in H$:
\[
p(xh)=p(x)\;,\quad\quad q(gx)=q(x)\;,\quad\quad (gx)h=g(xh)\;.
\]
Any left $G$-locale is automatically a $G$-$X/G$-bilocale with respect to the anchor map $\pi$ of \eqref{orbitlocale}.
We shall often use the notation ${_G X_H}$ to indicate that $X$ is a $G$-$H$-bilocale.
A \emph{map} of bilocales $f: {_G X_H}\to  {_G Y_H}$ is a map of locales that is both a map of left $G$-locales and a map of right $H$-locales. We shall denote the category of $G$-$H$-bilocales by $G$-$H$-$\Loc$.

\paragraph{Quantale modules.}

Given a unital involutive quantale $Q$, by a \emph{(left) $Q$-module} will be meant a sup-lattice $X$ equipped with an associative unital left action $Q\otimes X\to X$ in $\SL$ (the involution of $Q$ plays no role). The action of an element $a\in Q$ on $x\in X$ is denoted by $ax$. By a \emph{homomorphism} of left $Q$-modules $h:X\to Y$ is meant a $Q$-equivariant homomorphism of sup-lattices. Analogous definitions apply to right modules.

Let $G$ be an \'etale groupoid. In order to simplify notation we shall write $Q$ instead of $\opens(G)$. The left
$Q$-module associated to an action $(X, p, \act)$ is denoted simply by $X$ (rather
than $\opens(X)$ as in \cite{GSQS}), and we refer to it as a \emph{(left) $Q$-locale},
by which is meant a locale $X$ that is also a unital left $Q$-module satisfying the
\emph{anchor condition}
\[ bx = b1_X\wedge x \]
for all $b\in Q_0$ and $x\in X$.

The \emph{category of left $Q$-locales} has the left $Q$-locales as objects, and the
morphisms are the maps of locales whose inverse images are homomorphisms
of left $Q$-modules. This category is denoted by $Q$-$\Loc$ and it is isomorphic to $G$-$\Loc$~\cite{GSQS}*{Theorem 3.21}.

If $X$ is a left $Q$-locale, its action is a sup-lattice homomorphism
$Q\otimes X\to X$ that factors through another sup-lattice homomorphism $\alpha$ as
\[ \xymatrix{ Q\otimes X \ar@{->>}[r] & Q\otimes_{Q_0} X\ar[r]^-{\alpha} &X}\;.
\]
The right adjoint $\alpha_*$ is given by the following equivalent formulas:
\begin{eqnarray}
\alpha_*(x) &=& \V \{ a\otimes y\in Q\otimes X\st ay\leq x   \} \label{eq:rightadjformula1}\\
&=& \V_{s\in Q_\ipi} s\otimes s^*x\;. \label{eq:rightadjformula2}
\end{eqnarray}
The latter shows that $\alpha_*$ preserves arbitrary joins.  The corresponding group\-oid
action $\act : G\otimes_{G_0} X\to X$ is given by $\act^* =\alpha_*$. 
Let $G$ be an \'etale groupoid with quantale $Q = \opens(G)$, and
$X$ a left $G$-locale. Recall that an element $x \in X$ is \emph{invariant} if the following equivalent
conditions hold (for $X$ regarded as a $Q$-module):
\begin{enumerate}
\item For all $a \in Q$ we have $ax \leq x$\;;
\item For all $s \in Q_\ipi$ we have $sx \leq x$\;;
\item $1x \leq x$\;;
\item $1x = x$\;.
\end{enumerate}
The orbit locale $X/G$ coincides with the set $I(X)$ of invariant elements of the action \cite{Re15}*{Theorem 3.3}.
We conclude this paragraph by looking at a few simple properties of $Q$-locales. For all $s\in Q_\ipi$ and $x, y\in X$ we have:
\begin{align}
s(x \wedge y) &= sx \wedge sy\;;\\
s(x \wedge s^*y) &= sx \wedge y\label{partialunitaction}\;.
\end{align}

\paragraph{Quantale bimodules.}

Let $Q$ and $R$ be unital involutive quantales. By a \emph{$Q$-$R$-bimodule} is meant a sup-lattice $X$ with structures of unital left $Q$-module and unital $R$-module that satisfy the following \emph{associativity condition} for all $a\in Q$, $x\in X$ and $r\in R$:
\[
(ax)r= a(xr)\;.
\]

Let $Q$ and $R$ be inverse quantal frames. A \emph{$Q$-$R$-bilocale} is a $Q$-$R$-bimodule $X$ that is also a locale satisfying the \emph{left and right anchor conditions} for all $b\in Q_0$, $c\in R_0$, and $x\in X$:
\[
bx=b1_X\wedge x\quad\text{and}\quad xc=1_Xc\wedge x\;.
\]
Any left $Q$-locale is a $Q$-$I(X)$-bilocale with respect to the right action of $I(X)$ defined by $(x,y)\mapsto xy:=x\wedge y$ for all $x\in X$ and $y\in I(X)$. This follows from the analogous fact for groupoids but can also be easily verified using \eqref{partialunitaction}, since for all $s\in Q_\ipi$, $x\in X$ and $y\in I(X)$ we have
\[
sx\wedge y = s(x\wedge s^*y)\le s(x\wedge y)\le sx\wedge sy\le sx\wedge y\;,
\]
and thus $(sx)y=s(xy)$.

A \emph{map} of $Q$-$R$-bilocales $f:X\to Y$ is a map of locales whose inverse image $f^*$ is a homomorphism of $Q$-$R$-bimodules. The resulting category is denoted by $Q$-$R$-$\Loc$.

Let $G$ and $H$ be \'etale groupoids, and let $X$ be a $G$-$H$-bilocale. Then $X$ is an $\opens(G)$-$\opens(H)$-bilocale, and the categories $G$-$H$-$\Loc$ and $\opens(G)$-$\opens(H)$-$\Loc$ are isomorphic~\cite{Re15}.

Bilocales behave well with respect to tensor products, in the sense that if $Q$, $R$, and $S$ are inverse quantal frames then the tensor product $X\otimes_R Y$ of bilocales ${_Q X_R}$ and ${_R Y_S}$ is a $Q$-$S$-bilocale.
So we obtain a bicategory of quantale bilocales which is bi-equivalent to the bicategory of groupoid bilocales~\cite{Re15}.

We also note that the categories $Q$-$R$-$\Loc$ and $R$-$Q$-$\Loc$ are isomorphic: to each $Q$-$R$-bilocale $X$ the isomorphism assigns the \emph{dual} $R$-$Q$-bilocale $X^*$ which coincides with $X$ as a locale, and whose left $R$-action and right $Q$-action are defined, for all $a\in Q$, $r\in R$, and $x\in X$, by
\[
x\cdot a := a^*x\quad\quad\textrm{and}\quad\quad r\cdot x := xr^*\;.
\]
Similarly, due to the isomorphism $G$-$H$-$\Loc\cong \opens(G)$-$\opens(H)$-$\Loc$, if $G$ and $H$ are \'etale groupoids and $X$ is a $G$-$H$-bilocale  we obtain an $H$-$G$-bilocale $X^*$. The actions of $G$ and $H$ on $X^*$ coincide with the actions on $X$ composed with the inversion maps of the groupoids in the obvious way.

The following proposition, whose proof is straightforward, relates the isomorphism $Q$-$R$-$\Loc\cong R$-$Q$-$\Loc$ to the tensor product:

\begin{lemma}\label{commutativityoftensor}
Let $Q,R$ and $S$ be inverse quantal frames. Suppose that $X$ is a $Q$-$R$-bilocale and $Y$ is an $R$-$S$-bilocale. Then there is an isomorphism of $S$-$Q$-bilocales
\[  (X\otimes_R Y)^* \cong Y^*\otimes_R X^* \;.   \]
\end{lemma}

\subsection{Sheaves and Hilbert modules}

Let us conclude this preliminary section by recalling some facts about Hilbert modules and sheaves.

\paragraph{Hilbert modules.}
Let $Q$ be a unital involutive quantale. By a \emph{pre-Hilbert $Q$-module}~\cite{Paseka} is meant a left $Q$-module $X$ equipped with a binary operation $\langle -,-\rangle: X\times X\to Q$,
called the \emph{inner product}, which for all $x, x_i, y \in X$ and $a \in Q$ satisfies the
following axioms:
\begin{align}
\langle ax,y\rangle &= a\langle x,y\rangle\\
\bigl\langle\V_i x_i,y \bigr\rangle &= \V_i \langle x_i,y \rangle\\
\langle x,y\rangle &=\langle y,x \rangle^*\;.
\end{align}
It follows that the inner product is right skew-linear:
\[
\langle x,ay\rangle = \langle x,y\rangle a^*\;.
\]
The pre-Hilbert module is said to be \emph{full} if the inner product is surjective:
\[\langle X,X\rangle = Q\;.\]
By a \emph{Hilbert $Q$-module} is meant a pre-Hilbert $Q$-module whose inner
product is \emph{non-degenerate}:
\begin{align}
\langle x,-\rangle = \langle y,-\rangle \Rightarrow x=y\;.
\end{align}
Let $Q$ be a unital involutive quantale and let $X$ be a pre-Hilbert
$Q$-module. Any set $\sections \subset X$ such that
\[
x=\V_{t\in \sections} \langle x,t \rangle t
\]
for all $x \in X$ is called a \emph{Hilbert
basis}. We say that the pre-Hilbert module $X$ is \emph{complete} if it has a Hilbert basis. In this case $X$ is a Hilbert module, and Parseval's identity holds for all $x,y\in X$:
\[
\langle x,y\rangle = \V_{s\in\sections}\langle x,s\rangle\langle s,y\rangle\;.
\]
The largest Hilbert basis of a complete Hilbert module consists of all the \emph{Hilbert sections}, which are the elements $s \in X$ such that
$\langle x, s\rangle s \leq x$ for all $x \in X$. If $s$ is a Hilbert section and $x,y\in X$ then $\langle x,s\rangle\langle s,y\rangle\le\langle x,y\rangle$. A Hilbert section $s$ will be said to be \emph{regular} if $\langle s,s\rangle s=s$.

\begin{lemma}\label{subhilbertbasis0}
Let $Q$ be a unital involutive quantale and $X$ a Hilbert $Q$-module. Any set of regular Hilbert sections which is join-dense in $X$ is a Hilbert basis.
\end{lemma}

\begin{proof}
Let $\sections\in X$ be a set satisfying the stated conditions and let $x\in X$. Then $x=\V_i s_i$ for some family $(s_i)$ in $\sections$, and thus
\[
\V_{s\in\sections} \langle x,s\rangle s=\V_i\V_{s\in\sections}\langle s_i,s\rangle s=\V_i s_i =x\;.
\]
\end{proof}

\paragraph{Sheaves.}

Now we recall some results from~\cites{RR,GSQS}.
Let $B$ be a locale. A $B$-locale is said to be \emph{open} if its projection map $p$ is an open map of locales. Let $X$ be a locale which is also a $B$-module (not necessarily satisfying the anchor condition). Then $X$ is an open $B$-locale if and only if there exists a monotone equivariant map $\spp_X:X\to B$, called the \emph{support} of $X$, such that 
\begin{equation}\label{sppcondition}
\spp_X(x)x=x
\end{equation}  
for all $x\in X$.

\begin{example}\label{exm:openIXlocale}
Let $Q$ be an inverse quantal frame and $X$ a left $Q$-locale. As observed earlier, this is automatically a $Q$-$I(X)$-bilocale with the right action of $I(X)$ defined by binary meet. Moreover, $X$ is open as a right $I(X)$-locale if and only if for all $x,x'\in X$ we have
\begin{equation}\label{opennessIX}
1_Qx\wedge 1_Qx'\le 1_Q(x\wedge 1_Q x')\;,
\end{equation}
in which case the support $\tspp:X\to I(X)$ is given by $\tspp(x)=1_Qx$ for all $x\in X$. Indeed, $\tspp$ satisfies the right module version of \eqref{sppcondition} because for all $x\in X$ we have $x\tspp(x)=x\wedge 1_Q x=x$, and for all $x\in X$ and $y\in I(X)$ we have
\[
\tspp(xy)=1_Q(x\wedge y)\le 1_Qx\wedge 1_Qy\le 1_Qx\wedge y=\tspp(x)y\;,
\]
so $\tspp$ is $I(X)$-equivariant if and only if \eqref{opennessIX} holds (make $y=1_Qx'$).
\end{example}

Let $X$ be an open $B$-locale. The local sections of its projection $p:X\to B$ can be described in module-theoretic language as being the elements $s\in X$ such that for all $x\leq s$ we have
\begin{equation}\label{localsection}
\spp_X(x)s=x\;.
\end{equation}
We shall denote the set of all the local sections of $X$ by $\sections_X$. Therefore $p$ is a local homeomorphism if and only if
\[   \V\sections_X = 1_X\;. \]
In this case we say that $X$ is an \emph{\'etale $B$-locale}, or a \emph{$B$-sheaf}.
If $G$ is an \'etale groupoid, a \emph{$G$-sheaf} is defined to be a $G$-locale whose projection is a $G_0$-sheaf.

Now let $Q$ be an inverse quantal frame. By
a \emph{$Q$-sheaf} (resp.\ an \emph{open $Q$-locale}) is meant a $Q$-locale $X$ such that the induced action of $Q_0$ on $X$ defines a $Q_0$-sheaf (resp.\ an open $Q_0$-locale). If $G$ is an \'etale groupoid, the $G$-sheaves can be identified with the $\opens(G)$-sheaves. This result extends to morphisms of sheaves, so the classifying topos of $G$ is equivalent to the category of $\opens(G)$-sheaves.

The \emph{local sections} of a $Q$-sheaf $X$ are defined to be the local sections of
$X$ regarded as a $Q_0$-sheaf.
Moreover, complete Hilbert $Q$-modules and $Q$-sheaves are the same thing, and the local sections coincide with the Hilbert sections. The support of a $Q$-sheaf $X$ is given by
\begin{equation}
\spp_X(x) = \langle x, 1_X\rangle \wedge e = \langle x, x\rangle \wedge e\;. \label{supportformulamod}
\end{equation}
In particular, for a locale $B$ and a $B$-sheaf $X$ we have $\spp_X(x)=\langle x,1\rangle=\langle x,x\rangle$.

A $Q$-sheaf $X$ is an example of a \emph{stably supported module}, so the support satisfies the following conditions for all $b \in Q_0$, $a\in Q$ and $x,y \in X$:
\begin{align}
\spp_X(bx) &=b\wedge \spp_X(x)\\
\spp_X(bx) &= \spp_Q(b\spp_X(x))\\
\spp_Q(\langle x, y\rangle) &\leq \spp_X(x) = \spp_Q(\langle x, x\rangle) = \spp_Q(\langle x, 1_X\rangle)\label{eq:sppXsppQinner}\\
\spp_X(x)a  &= \langle x, 1_X \rangle \wedge a\;.
\end{align}

\begin{lemma}\label{subhilbertbasis}
Let $Q$ be an inverse quantal frame and $X$ a $Q$-sheaf.
\begin{enumerate}
\item\label{subhb0} Any Hilbert section of $X$ is regular.
\item\label{subhb1} Any join-dense set of Hilbert sections of $X$ is a Hilbert basis.
\item\label{subhb2} Any downwards closed set of Hilbert sections of $X$ such that $\V\sections=1_X$ is a Hilbert basis.
\end{enumerate}
\end{lemma}

\begin{proof}
\eqref{subhb0} follows from \eqref{supportformulamod}, for if $s$ is a Hilbert section we have $s=\spp_X(s)s=(\langle s,s\rangle\wedge e)s\le\langle s,s\rangle s\le s$. Then \eqref{subhb1} follows from Lemma~\ref{subhilbertbasis0}, and \eqref{subhb2} follows from \eqref{subhb1} because $X$ is a locale and thus any downwards closed cover of $1_X$ is join-dense.
\end{proof}

\section{Open maps and sheaves}\label{sec:oms}

In this section we prove (mostly new) technical results about sheaves on locales and quantales which, besides being of independent interest, will be needed later but are not found in \cites{RR,GSQS}.

\subsection{Open maps}

The following is a useful fact about projections of pullbacks of open maps of locales.

\begin{lemma}\label{directimage}
Let $X$ and $Y$ be open $B$-locales. Then the projections
\[
\pi_1: X\otimes_B Y\to X\quad\quad\textrm{and}\quad\quad \pi_2:X\otimes_B Y\to Y
\]
are open maps whose direct image homomorphisms are given, for all $x\in X$ and $y\in Y$, by
\begin{equation}
(\pi_1)_!(x\otimes y)=\spp_Y(y)x\quad\quad\text{and}\quad\quad (\pi_2)_!(x\otimes y)=\spp_X(x)y\;.
\end{equation} 
\end{lemma}
\begin{proof}
Since $X$ and $Y$ are open $B$-locales, the maps $\pi_1$ and $\pi_2$ are pullbacks of open maps and thus they are open.
The mapping $X\times Y\to X$ given by $(x, y)\mapsto \spp_Y(y)x$ preserves joins in each variable separately, so it defines a sup-lattice homomorphism $\phi:X\otimes Y\to X$. Let us prove that it factors through the quotient $X\otimes Y\to X\otimes_B Y$. Indeed, for all $b\in B$ we have
\begin{align*}
\phi(bx\otimes y) &= \spp_Y(y)bx\\
&= b\spp_Y(y)x\\
&=\spp_Y(by)x\quad\quad \text{($\spp_Y$ is $B$-equivariant)}\\
&=\phi(x\otimes by)\;.
\end{align*}
So we have a well defined sup-lattice homomorphism $\phi':X\otimes_B Y\to X$ given by $\phi'(x\otimes y)=\spp_Y(y)x$.
Finally, let us prove that $\phi'$ is the left adjoint of the map $\pi^*_1(x)=x\otimes 1_Y$. The unit can be verified as follows:
\[ \pi^*_1(\phi'(x\otimes y))=\pi^*_1(\spp_Y(y)x)=\spp_Y(y)x\otimes 1_Y= x\otimes \spp_Y(y)1_Y\geq x\otimes \spp_Y(y)y=x\otimes y\;.  \]
And the co-unit is as follows:
\[ \phi'(\pi^*_1(x))=\phi'(x\otimes 1_Y)=\spp_Y(1_Y)x\leq x\;.  \]
This implies that $\phi'=(\pi_1)_!$. A similar argument works for $(\pi_2)_!$. 
\end{proof}

A simple consequence of the above lemma is the following well known fact about pullbacks of open surjections:

\begin{corollary}\label{lem:surjstab}
Let the following be a pullback diagram in $\Loc$ where $p$ is an open surjection and $q$ is open:
\[
\xymatrix{
X\otimes_B Y\ar[rr]^{\pi_2}\ar[d]_{\pi_1}&&Y\ar@{->>}[d]^p\\
X\ar[rr]_q&& B
}
\]
Then $\pi_1$ is an open surjection.
\end{corollary}

\begin{proof}
$\pi_1$ is open because of the pullback stability of open maps of locales. In order to show that $\pi_1$ is an open surjection we shall regard $X$ and $Y$ as open $B$-locales and apply Lemma~\ref{directimage}. Note that $\spp_Y(1_Y)=1_B$ because $p$ is an open surjection. Let $x,x'\in X$ be such that $\pi_1^*(x)=\pi_1^*(x')$. This means that $x\otimes 1_Y=x'\otimes 1_Y$ in $X\otimes_B Y$. Therefore
\[
x=1_B x=\spp_Y(1_Y) x=(\pi_1)_!(x\otimes 1_Y)
=(\pi_1)_!(x'\otimes 1_Y)=x'\;,
\]
so $\pi_1^*$ is injective. 
\end{proof}

\subsection{Sheaves on locales}

Let $B$ be a locale and $X$ a $B$-sheaf. 
By a \emph{compatible set} of local sections is meant a set $Z\subset \sections_X$ such that any two elements $s,t\in Z$ are compatible:
\[
\spp_X(s)t=\spp_X(t)s\;.
\]
Equivalently, $Z$ is compatible if and only if $\V Z$ is itself a local section.

\begin{theorem}\label{XOX}
Let $B$ be a locale and let $X$ and $Y$ be $B$-sheaves with Hilbert bases $\sections_X'$ and $\sections_Y'$, respectively. Then $X\otimes_B Y$ is a $B$-sheaf, and the following set is a Hilbert basis:
\[
\sections = \bigl\{s\otimes t\st s\in\sections'_X,\ t\in\sections'_Y\bigr\}\;.
\]
Moreover, if $\sections'_X=\sections_X$ and $\sections'_Y=\sections_Y$ then $\sections$ is closed under joins of compatible sets in $X\otimes_B Y$; that is, every element $\V_i s_i\otimes t_i$, with the elements $s_i\otimes t_i$ pairwise compatible, can be written in the form $s\otimes t$ with $s\in\sections_X$ and $t\in\sections_Y$.
\end{theorem}

\begin{proof}
Suppose that $X$ and $Y$ are $B$-sheaves, and let us prove that $X\otimes_B Y$ is a $B$-sheaf. The mapping $\phi:X\times Y\to B$ defined by
\[
(x,y)\mapsto\spp_X(x)\wedge\spp_Y(y)
\]
preserves joins in each variable, and for all $b\in B$ we have
\[
\phi(bx,y)=\phi(x,by)=b\wedge \spp_X(x)\wedge \spp_Y(y)
\]
because the supports $\spp_X$ and $\spp_Y$ are $B$-equivariant. Hence, there is a sup-lattice homomorphism
\[
\spp:X\otimes_B Y\to B
\]
defined for all $x\in X$ and $y\in Y$ by
\[
\spp(x\otimes y)=\spp_X(x)\wedge\spp_Y(y)
\]
which is obviously $B$-equivariant. Let $x,x'\in X$ and $y,y'\in Y$. We have
\begin{eqnarray}
\spp(x\otimes y)x\otimes y&=&x\otimes y \label{niceineq1}\\
\spp(x'\otimes y')x\otimes y&\le& x\otimes y\label{niceineq2}\;,
\end{eqnarray}
as the following derivations show:
\begin{eqnarray*}
\spp(x'\otimes y')x\otimes y &=& (\spp_X(x')\wedge\spp_y(y')) x\otimes y
=\spp_X(x')x\otimes\spp_Y(y')y\le x\otimes y\;;\\
\spp(x\otimes y)x\otimes y &=& \spp_X(x)x\otimes\spp_Y(y)y= x\otimes y\;.
\end{eqnarray*}
Let $\V_{i\in I} x_i\otimes y_i$ be an arbitrary element of $X\otimes_B Y$. Then using \eqref{niceineq1} and \eqref{niceineq2} we obtain
\[
\spp\bigl(\V_{i\in I} x_i\otimes y_i\bigr)\V_{i\in I} x_i\otimes y_i =
\V_{j\in I}\V_{i\in I} \spp(x_j\otimes y_j)x_i\otimes y_i =\V_{i\in I} x_i\otimes y_i\;,
\]
and this shows that $X\otimes_B Y$ is an open $B$-locale. Now let $x\in X$, $y\in Y$, $s\in\sections'_X$ and $t\in\sections'_Y$ be such that
\[
x\otimes y\le s\otimes t\;.
\]
Applying Lemma~\ref{directimage} we obtain the inequality $\spp_Y(y) x\le \spp_Y(t) s$:
\[
\spp_Y(y) x = (\pi_1)_!(x\otimes y)\le (\pi_1)_!(s\otimes t) = \spp_Y(t) s\;.
\]
Therefore, since $\spp_Y(t) s\in\sections_X$, we obtain
\[
\spp_Y(y) x = \spp_X\bigl(\spp_Y(y)x\bigr)\spp_Y(t)s = \bigl(\spp_X(x)\wedge\spp_Y(y)\wedge\spp_Y(t)\bigr)s\;.
\]
Similarly, an application of $(\pi_2)_!$ yields the following equality:
\[
\spp_X(x) y = \bigl(\spp_X(x)\wedge\spp_Y(y)\wedge\spp_X(s)\bigr)t\;.
\]
Furthermore, we have
\[
s\otimes t = \spp_X(s)s\otimes\spp_Y(t)t = \spp_Y(t)s\otimes\spp_X(s)t\;,
\]
and thus
\begin{align*}
\spp(x\otimes y)s\otimes t &=
\bigl(\spp_X(x)\wedge\spp_Y(y)\bigr)\spp_Y(t)s\otimes \spp_X(s)t\\
&= \bigl(\spp_X(x)\wedge\spp_Y(y)\wedge\spp_Y(t)\bigr)s\otimes\bigl(\spp_X(x)\wedge\spp_Y(y)\wedge\spp_X(s)\bigr)t\\
&=\spp_Y(y)x\otimes\spp_X(x)y = x\otimes y\;.
\end{align*}
Hence, for all elements
$\xi:=\V_i x_i\otimes y_i\in X\otimes_B Y$ we have
\[
\xi\le s\otimes t\Rightarrow \spp(\xi)s\otimes t=\xi\;,
\]
which shows that $s\otimes t$ is a local section of the open $B$-locale $X\otimes_B Y$. Moreover, it is clear that $\V \sections=1_X\otimes 1_Y$, and thus $X\otimes_B Y$ is a $B$-sheaf with Hilbert basis $\sections$.

Now suppose that $\sections'_X=\sections_X$ and $\sections'_Y=\sections_Y$ and that $(s_i\otimes t_i)_{i\in I}$ is a pairwise compatible family in $\sections$; that is, for all $i,j\in I$ we have
\[
\spp(s_i\otimes t_i)(s_j\otimes t_j) = \spp(s_j\otimes t_j)(s_i\otimes t_i)\;.
\]
Equivalently,
\[
\spp_{X}(s_j)\spp_Y(t_j)(s_i\otimes t_i) = \spp_{X}(s_i)\spp_Y(t_i)(s_j\otimes t_j)\;,
\]
which in turn is equivalent to
\[
\spp_{X}(s_j)s_i\otimes \spp_Y(t_j)t_i = \spp_{X}(s_i)s_j\otimes \spp_Y(t_i)t_j\;.
\]
Since for all $s\in\sections_X$ and $t\in\sections_Y$ we have $s\otimes t=\spp_Y(t)s\otimes\spp_X(s)t$, we shall assume, without loss of generality, that $\spp_X(s_i)=\spp_Y(t_i)$ for all $i\in I$.
Then, applying $(\pi_1)_!$ and $(\pi_2)_!$ to the latter equation (again using Lemma~\ref{directimage}), for all $i,j\in I$ we obtain
\[
\spp_X(s_j)s_i=\spp_X(s_i)s_j\quad \text{and}\quad \spp_Y(t_j)t_i=\spp_Y(t_i)t_j\;,
\] 
\ie, both $(s_i)$ and $(t_i)$ are compatible families. Therefore we obtain
\[
\bigl(\V_i s_i\bigr)\otimes\bigl(\V_i t_i\bigr)\in \sections\;.
\]
Finally, let us prove that $\bigl(\V_i s_i\bigr)\otimes\bigl(\V_i t_i\bigr)=\V_i s_i\otimes t_i$. For all $i,j\in I$ we have $s_j\otimes t_i\leq s_i\otimes t_i$ due to compatibility:
\[
s_j\otimes t_i = s_j\otimes \spp_Y(t_i)t_i
= \spp_Y(t_i)s_j\otimes t_i
= \spp_X(s_i)s_j\otimes t_i
\leq s_i\otimes t_i\;.
\]
Therefore $\bigl(\V_i s_i\bigr)\otimes\bigl(\V_i t_i\bigr)= \V_j\V_{i} s_j\otimes t_i =\V_i s_i\otimes t_i$.
\end{proof}

\begin{lemma}\label{liftingloc}
Let $B$ be a locale and $X$ a $B$-sheaf.
For all (left) $B$-modules $M$ and all mappings
$h:\sections_X\to M$ that are $B$-equivariant and preserve joins of compatible sets of local sections there is a unique homomorphism of $B$-modules $h^\sharp:X\to M$ that extends $h$.
\end{lemma}

\begin{proof}
An analogous universal property holds for any inverse quantal frame $Q$, with $Q_\ipi$ playing the role of $\sections_X$. The details of that can be found in~\cite{Re07}, but we summarize the analogous argument for $B$-sheaves here.
Denote by $\lcc(\sections_X)$ the set of all the sets $I\subset \sections_X$ which are downwards closed in $\sections_X$ and such that for every compatible subset $S\subset I$ we have $\V S\in I$. Denote by $\eta:\sections_X\to\lcc(\sections_X)$ the mapping that to each local section $s\in\sections_X$ assigns its principal order ideal $\downsegment(s)$. Using the fact that both the action and the binary meets distribute over joins of compatible subsets of $\sections_X$ it is straightforward to verify that for all left $B$-modules $M$ and all mappings
$h:\sections_X\to M$ that are $B$-equivariant and preserve joins of compatible sets of local sections there is a unique homomorphism of left $B$-modules $h':\lcc(\sections_X)\to M$ such that for all $s\in\sections_X$ we have
\[
h'(\downsegment(s))=h(s)\;.
\]
Evidently, $h'$ is defined, for each $I\in\lcc(\sections_X)$, by
$h'(I)=\V h(I)$.
Let $h:\sections_X\to X$ be the inclusion. Then $h':\lcc(\sections_X)\to X$ is a homomorphism of left $B$-modules given by $h'(I)=\V I$ for all $I\in\lcc(\sections_X)$. It can be easily verified that it is also a homomorphism of locales, moreover one which is injective on the set $\sections'_X$ of principal order ideals generated by elements of $\sections_X$. Since $h$ is injective and $\sections'_X$ is a downwards closed basis of $\lcc(\sections_X)$ we conclude, by~\cite{Re07}*{Prop.\ 2.2}, that $h'$ is itself injective, so we have $X\cong\lcc(\sections_X)$, and thus $X$ has the required universal property. 
\end{proof}

\begin{theorem}\label{directimagepairing}
Let $X$, $Y$, and $Z$ be $B$-sheaves, and let $f$ and $g$ be sheaf homomorphisms:
\[
\xymatrix{&Z\ar[rd]^g\ar[dl]_f\ar@{.>}[d]|{\langle f,g\rangle}\\
X&X\otimes_B Y\ar[l]^-{\pi_1}\ar[r]_-{\pi_2}&Y}
\]
Then $\langle f,g\rangle$ is itself a sheaf homomorphism, and its direct image is given by, for all sections $s\in \sections_Z$,
\begin{equation}\label{eqdirimage}
\langle f,g\rangle_!(s) = f_!(s)\otimes g_!(s)\;.
\end{equation}
\end{theorem}

\begin{proof}
The pullback $X\otimes_B Y$ exists in the slice category $\Loc/B$ and thus it exists in the category of $B$-locales. The pairing map is itself a map in $\Loc/B$, and thus it is a sheaf homomorphism because both $f$ and $\pi_1$ are; its direct image sends sections of $Z$ to sections of $X\otimes_B Y$, and we have $\spp_{X\otimes Y}(\langle f,g\rangle_!(s)) = \spp_Z(s)$ for all $s\in\sections_Z$. Let then $\varphi:\sections_Z\to X\otimes_B Y$ be the mapping defined by
\[
\varphi(s) =  f_!(s)\otimes g_!(s)
\]
for all $s\in\sections_Z$. First we verify that this map preserves joins of compatible subsets and that therefore it extends uniquely to a sup-lattice homomorphism $\varphi^\sharp:Z\to X\otimes_B Y$. Indeed, let $(s_i)$ be a compatible family of sections in $\sections_Z$. Then we have
\begin{eqnarray*}
\varphi\bigl(\V_i s_i\bigr) &=& f_!\bigl(\V_i s_i\bigr)\otimes g_!\bigl(\V_i s_i\bigr)=\V_{j,k} f_!(s_j)\otimes g_!(s_k)\\
&=&\V_{j,k} f_!(\spp_Z(s_j)s_j)\otimes g_!(\spp_Z(s_k) s_k)
=\V_{j,k} \spp_Z(s_j)f_!(s_j)\otimes \spp_Z(s_k) g_!(s_k)\\
&=&\V_{j,k} \spp_Z(s_k)f_!(s_j)\otimes \spp_Z(s_j) g_!(s_k)=\V_{j,k} f_!(\spp_Z(s_k)s_j)\otimes g_!(\spp_Z(s_j) s_k)\\
&=&\V_i f_!(s_i)\otimes g_!(s_i)\quad\quad (\text{because }\spp_Z(s_k)s_j=\spp_Z(s_j)s_k)\\
&=&\V_i\varphi(s_i)\;.
\end{eqnarray*}
The inverse image $\langle f,g\rangle^*$ is the copairing $\lbrack f^*,g^*\rbrack$, so in order to prove \eqref{eqdirimage} we shall prove that $\varphi^\sharp$ is left adjoint to $\lbrack f^*,g^*\rbrack$, for which it suffices to verify the unit and counit conditions of the adjunction. We begin with the unit. Let $s\in\sections_Z$. Then
\[
s\le f^*f_!(s)\wedge g^* g_!(s)=\lbrack f^*,g^*\rbrack(\varphi(s))\;.
\]
Hence, for all $z\in Z$ we have $z\le \lbrack f^*,g^*\rbrack(\varphi^\sharp(z))$. For the co-unit we have
\begin{align*}
\varphi^\sharp([f^*,g^*](x\otimes y)) &= \varphi^\sharp(f^*(x)\wedge g^*(y))\\
&=\V_{\substack{s\in \sections_Z\\ s\leq f^*(x)\\s\leq f^*(y)}}\varphi(s)\\
&=\V_{\substack{s\in \sections_Z\\ s\leq f^*(x)\\s\leq f^*(y)}}f_!(s)\otimes g_!(s)\\
&\leq f_!(f^*(x))\otimes g_!(g^*(y))\\
&\leq x\otimes y\;.
\end{align*}
Therefore, for all $\xi\in X\otimes_B Y$ we have $\varphi^\sharp([f^*,g^*](\xi))\leq \xi$. 
\end{proof}

\subsection{Sheaves on inverse quantal frames}\label{sec:shqu}

The following theorem gives a useful formula for computing inner products of quantale sheaves.

\begin{theorem}\label{innerproduct2}
Let $Q$ be an inverse quantal frame and $X$ be a complete Hilbert $Q$-module. Then for all $x,y\in X$
\begin{equation}\label{innerproduct1}
\langle x,y\rangle=\V_{u\in Q_{\mathcal{I}}} u \spp_X(u^*x \wedge y)\;.
\end{equation}
\end{theorem}

\begin{proof}
Since $\sections_X$ is join-dense in $X$ it suffices to prove the assertion for sections $s,t\in \sections_X$: 
\begin{equation}\label{b_!}
\langle s, t\rangle=\V_{u\in Q_{\mathcal{I}}} u \spp_X(u^*s \wedge t)\;.
\end{equation}
Let us prove now that $\V_{u\in Q_{\mathcal{I}}} u \spp_X(u^*s \wedge t)$ coincides with the inner product induced by $X$ (\cf\   \cite{GSQS}*{Theorem 4.55}):
\begin{equation}\label{innerproduct}
\langle s,t\rangle = \V \bigl\{u\in Q_{\mathcal{I}}\st\spp_Q(u)\leq \spp_X(s),\  \spp_Q(u^*)\leq \spp_X(t) \text{ and } ut\leq s  \bigr\}\;.
\end{equation}
Clearly, $u\spp_X(u^*s\wedge t)\in Q_{\mathcal{I}}$. Now, we must show that $u\spp_X(u^*s\wedge t)$ satisfies the conditions of \eqref{innerproduct}. In fact,
\begin{align*}
\spp_Q(u\spp_X(u^*s\wedge t)) & \leq \spp_Q(u\spp_X(u^*s))\\
&= \spp_Q(u\spp_Q(u^*\spp_X(s)))& (\text{$\spp_X$ is stable})\\
&= \spp_Q(u(u^*\spp_X(s))(u^*\spp_X(s))^*)\\
&= \spp_Q(uu^*\spp_X(s)u)\\
&= \spp_Q(\spp_X(s)u)\\
&= \spp_X(s)\spp_Q(u)& (\text{$\spp_Q$ is stable})\\
&\leq \spp_X(s)\;. 
\end{align*}
On the other hand, we have:
\begin{align*}
\spp_Q\bigl((u\spp_X(u^*s\wedge t))^*\bigr) &\leq \spp_Q((u\spp_X(t))^*)\\
&= \spp_Q(\spp_X(t)u^*)\\
&= \spp_X(t)\spp_Q(u^*)\quad (\text{$\spp_Q$ is stable})\\
&\leq \spp_X(t)\;.
\end{align*}
Finally, since $t$ is a local section, we obtain $u\spp_X(u^*s\wedge t)t\leq s$:
\[
u\spp_X(u^*s\wedge t)t = u(u^*s\wedge t)
= uu^*s\wedge ut
\leq s\;.
\]
Therefore
\[
\V_{u\in Q_{\mathcal{I}}} u \spp_X(u^*s \wedge t)\leq \langle s,t\rangle\;.
\]
In order to prove the converse assume that $u\in Q_\ipi$ satisfies the three conditions $\spp_Q(u)\leq \spp_X(s)$, $\spp_Q(u^*)\leq \spp_X(t)$, and $ut\leq s$. Then
\begin{align*}
u\spp_X(u^*s\wedge t) &= u\spp_X(u^*(s\wedge ut))\\
&= u\spp_X(u^*ut)& \text{(because $ut\leq s$)}\\
&= uu^*u\spp_X(t)& \text{(because $\spp_X$ is stable)}\\
&= u\spp_X(t)& \text{(because $uu^*u=u$)}\\
&=u\,.& \text{(because $\spp_Q(u^*)\le\spp_X(t)$)}
\end{align*}
Hence, $\langle s,t\rangle\leq \V_{u\in Q_{\mathcal{I}}} u \spp_X(u^*s \wedge t)$, so the equality holds.
\end{proof}

\subsection{Quantale bi-locales and sheaves}

Let $Q$ and $R$ be inverse quantal frames, and let $X$ be a $Q$-$R$-bilocale which is an open $Q$-locale and is also an open $R$-locale. From now on we shall denote by $\tspp$ the direct image of the right projection $X\to R_0$. The direct image of the left projection $X\to Q_0$ will be denoted by $\spp_X$ as before.

The following lemma is a quantale version of the invariance of the right projection (resp.\ left projection) under the left action (resp.\ right action) of a groupoid bilocale.

\begin{lemma}\label{tut}
Let $Q$ and $R$ be inverse quantal frames, and let $X$ be a $Q$-$R$-bilocale which is an open $Q$-locale and also an open $R$-locale. Then for all $a\in Q$ and $x\in X$ we have
\begin{eqnarray}
\tspp(ax)&=&\tspp(\spp_Q(a^*)x)\;, \label{eq:tut} \\
\spp_X(xr)&=&\spp_X(x\spp_R(r))\;. \label{eq:tut2}
\end{eqnarray}
\end{lemma} 

\begin{proof}
Since $X$ is a $Q$-$R$-bilocale we have
\[
\tspp\circ \act_!=\tspp\circ (\pi_2)_!\;,
\]
where $(\pi_2)_!$ is the direct image homomorphism of the map
$\pi_2:G_1\otimes_{G_0} X\to X$.  
By Lemma~\ref{directimage}, $(\pi_2)_!$ is given, for all  $a\in Q$ and $x\in X$, by
\begin{equation}\label{directimagepi2}
 (\pi_2)_!(a\otimes x)=\spp_Q(a^*)x\;.
 \end{equation}
Then $\tspp(ax) =  \tspp(\act_!(a\otimes x))=  \tspp((\pi_2)_!( a\otimes  x))=  \tspp(\spp_Q( a^*) x)$ because $X$ is an open $Q$-locale. This proves \eqref{eq:tut}, and \eqref{eq:tut2} is similar.
\end{proof}

The following theorem shows that for a suitable $Q$-$R$-bilocale the involute of an element $r\in R$ can be regarded as an adjoint operator.

\begin{theorem}\label{axyxay}
Let $Q$ and $R$ be inverse quantal frames, and let $X$ be a $Q$-$R$-bilocale which is a $Q$-sheaf and also an open $R$-locale. Then for all $r\in R$ and all $x,y\in X$ we have
\begin{equation}\label{eq:axyxay}
\langle xr^*,y\rangle = {\langle x,yr\rangle}\;.
\end{equation}
\end{theorem}

\begin{proof}
It suffices to prove \eqref{eq:axyxay} for $r\in R_\ipi$ and $x,y\in \sections_X$. Using Theorem~\ref{innerproduct2}, we have
\begin{align*}
\langle x,yr\rangle &= \V_{u\in R_\ipi} u\spp_X(u^*x\wedge yr)\\
&= \V_{u\in R_\ipi} u\spp_X(u^*x\wedge yrr^*r)\\
&=\V_{u\in R_\ipi} u\spp_X(u^*xr^*r\wedge yr)\\
&= \V_{u\in R_\ipi} u\spp_X\bigl((u^*xr^*\wedge y)r\bigr)\\
&= \V_{u\in R_\ipi} u\spp_X\bigl((u^*xr^*\wedge y)\spp_R(r)\bigr)&\text{(by Lemma~\ref{tut})}\\
&= \V_{u\in R_\ipi} u\spp_X\bigl(u^*xr^*\spp_R(r)\wedge y\bigr)\\
&= \V_{u\in R_\ipi} u\spp_X(u^*xr^*\wedge y) & \text{(because $r^*=r^*\spp_R(r)$)}\\
&={\langle xr^*,y\rangle}\;. 
\end{align*}
\end{proof}

\section{Hilsum--Skandalis maps}\label{sec:principalbundles}

In this section we provide an overview of definitions and facts concerning principal bundles, Hilsum--Skandalis maps and Morita equivalence for localic \'etale groupoids, in particular addressing the facts and definitions that will be adapted for inverse quantal frames in the remainder of the paper.

\subsection{Principal bundles}\label{subsec:principalbundles1}

Let $G$ be a groupoid and $M$ a locale. A \emph{(left) $G$-bundle $(X,p,\act,\pi)$ over $M$} is a (left) $G$-locale $(X, p,\act)$ equipped with a map of locales $\pi: X\to M$ such that the diagram
\[
\vcenter{\xymatrix{
G_1\otimes_{G_0} X\ar[r]^-{\act} \ar[d]_{\pi_2}  &X\ar[d]^{\pi}\\
X\ar[r]_{\pi} & M
}
}
\]
is commutative in $\Loc$. Any left $G$-locale $X$ is a $G$-bundle over $X/G$. 

A $G$-bundle over $M$ is \emph{principal} if the above diagram is a pullback diagram (equivalently,
\begin{equation}\label{principaliso}
\langle \act, \pi_2 \rangle: G_1\otimes_{G_0} X \to X\otimes_{M} X
\end{equation}
is an isomorphism of locales) and $\pi:X\to M$ is an open surjective map (necessarily a local homeomorphism because $G$ is \'etale --- see Lemma~\ref{funct2, lemma: covering} below).
We shall usually denote the inverse of the isomorphism $\langle \act,\pi_2\rangle$ of \eqref{principaliso} by $\langle \theta,\pi_2\rangle$, 
\begin{equation}\label{eq:theta}
\xymatrix{
X\otimes_{M} X\ar@<1.0ex>[rr]_{\cong}^{ \langle \theta,\pi_2 \rangle}&& G_1\otimes_{G_0} X \ar@<1.0ex>[ll]^{ \langle \act,\pi_2 \rangle}\;,
}
\end{equation}
where $\theta:X\otimes_{M} X\to G_1$ is a map of locales.

Similarly, we define \emph{(right) $G$-bundles} and \emph{principal (right) $G$-bundles} (with pullback over the range map $r$) over $M$.

\begin{example}
$G_1$ is itself a left $G$-bundle over $G_0$ with $\pi= r$ and $p=d$, and a right $G$-bundle with $\pi= d$ and $p=r$. Both $G$-bundles are principal and the actions commute.
\end{example}

The following two lemmas state simple properties of principal bundles that are well known in the context of topological groupoids.

\begin{lemma}\label{McongX/G}
Let $G$ be a groupoid and $(X,p,\act,\pi)$ a principal $G$-bundle over a locale $M$. Then $M\cong X/G$.
\end{lemma}

\begin{proof}
Since $\pi:X\to M$ is an open surjection, it can be obtained as a coequalizer~\cite{JT}*{section V.4}:
\[\xymatrix{ X\otimes_{M} X\ar@<0.7ex>[r]^-{ \pi_1}\ar@<-0.7ex>[r]_-{ \pi_2}& X\ar@{->>}[r]^-{\pi} & M\;.
}
\]
By principality, $\pi$ is also the coequalizer of  
\[\xymatrix{ G_1\otimes_{G_0} X\ar@<0.7ex>[r]^-{ \act}\ar@<-0.7ex>[r]_-{ \pi_2}& X\;,
}
\]
and thus $M\cong X/G$\;. 
\end{proof}

From here on the expression ``principal $G$-bundle'' will implicitly mean a principal $G$-bundle over $X/G$.

\begin{lemma}\label{funct2, lemma: covering}
Let $G$ be an \'etale groupoid and $X\equiv(X,p,\act,\pi)$ be a principal $G$-bundle. Then $\pi$ is a local homeomorphism.
\end{lemma}

\begin{proof}
Notice that $\pi$ is open by definition of principal bundle. Hence, in order to prove that $\pi$ is a local homeomorphism it is enough to show that the diagonal map $\Delta: X\to X\otimes_{X/G} X$ is open~\cite{JT}*{Chapter V}. Equivalently, since $X$ is principal, let us prove that the map
\begin{equation*}
\langle u\circ p, \ident_X\rangle:X\to G_1\otimes_{G_0} X
\end{equation*}
is open. This is true because we have the following factorization of $\ident_X$,
\[
\xymatrix{
X\ar[rr]^-{\langle u\circ p, \ident_X\rangle}\ar@{=}[rd]&&G_1\otimes_{G_0} X\ar[dl]^{\pi_2}\\
&X
}
\]
where $\pi_2$ is a local homeomorphism because it is a pullback of $d$, so from \cite{elephant}*{Lemma C1.3.2(iii)} it follows that $\langle u\circ p, \ident_X\rangle$ is a local homeomorphism.
\end{proof}

\subsection{Point-set reasoning}\label{sec:pointset}

Let $G$ be an \'etale groupoid and $X$ a principal $G$-bundle with action $\act$. Let also $\theta:X\otimes_{X/G} X\to G_1$ be the map such that $\langle\theta,\pi_2\rangle=\langle\act,\pi_2\rangle^{-1}$ [\cf\ \eqref{eq:theta} in section~\ref{subsec:principalbundles1}]. Note that, if $G$ is a topological groupoid and $X$ a topological principal $G$-bundle, $\theta(x,x')$ is, given some pair $(x,x')\in X\times_{X/G} X$, the unique arrow $g\in G_1$ such that $g x'=x$ (where we write $gx$ instead of $\act(g,x)$ for the action of an arrow $g\in G_1$ on an element $x\in X$). It follows, in particular, that $\theta(x,x)$ is the identity arrow $u(p(x))\in G_1$. Such properties are less easy to describe in the case of localic groupoids, but we obtain similar statements by replacing elements such as $x\in X$ by maps $\boldsymbol x$ into $X$, and expressions such as $\theta(x,x')$ by $\theta\circ\langle\boldsymbol x,\boldsymbol x'\rangle$, as follows:

\begin{lemma}\label{lem:localicreasoningwithglobalpoints}
Let $G$ be a groupoid and $(X,p,\act,\pi)$ a principal $G$-bundle. For all locales $Z$, and for all maps $\boldsymbol g:Z\to G_1$ and $\boldsymbol x,\boldsymbol x':Z\to X$ such that $r\circ\boldsymbol g=p\circ\boldsymbol x$ and $\pi\circ\boldsymbol x=\pi\circ\boldsymbol x'$, we have
\begin{enumerate}
\item\label{lem:localicreasoningwithglobalpoints1} $\theta\circ \langle \boldsymbol x, \boldsymbol x' \rangle=i\circ \theta\circ \langle \boldsymbol x', \boldsymbol x \rangle$.
\item\label{lem:localicreasoningwithglobalpoints2} $p\circ\boldsymbol x'=r\circ\theta\circ\langle\boldsymbol x,\boldsymbol x'\rangle$ (so $\langle\theta\circ\langle\boldsymbol x,\boldsymbol x'\rangle,\boldsymbol x'\rangle$ is a map into $G_1\otimes_{G_0} X$);
\item\label{lem:localicreasoningwithglobalpoints3} $\act\circ\langle\theta\circ\langle\boldsymbol x,\boldsymbol x'\rangle,\boldsymbol x'\rangle =\boldsymbol x$;
\item\label{lem:localicreasoningwithglobalpoints4} $\theta\circ\langle\boldsymbol x,\boldsymbol x\rangle=u\circ p\circ\boldsymbol x$.
\item\label{lem:localicreasoningwithglobalpoints5} $\theta\circ \langle \act\circ \langle \boldsymbol g, \boldsymbol x \rangle, \boldsymbol x' \rangle=m\circ \langle  \boldsymbol g, \theta\circ \langle \boldsymbol x, \boldsymbol x' \rangle \rangle$.
\end{enumerate}
\end{lemma}

\begin{proof}
1.  By Lemma \ref{G-loc,isom} the categories of left $G$-locales and right $G$-locales are isomorphic. Moreover, due to principality the upper rectangle of the following diagram is commutative, and thus so is the lower rectangle:
\[\xymatrix{ 
G_1\otimes_{G_0} X\ar[rr]^{\cong} \ar@/_4pc/[dd]_{\pi_1}   && X\otimes_{G_0} G_1 \ar@/^4pc/[dd]^{\pi_2}\\
X\otimes_{{X/G}} X\ar[u]^{\langle \theta,\pi_2\rangle}_{\cong} \ar[d]^{\theta}\ar[rr]^{\langle \pi_2,\pi_1 \rangle}_\cong  &&  X\otimes_{{X/G}} X\ar[u]_{\langle \pi_1,\theta\rangle}^{\cong}\ar[d]^{\theta}\\ 
G_1\ar[rr]_{i}^{\cong} && G_1
}
\]
Hence,
\begin{align*}
\theta\circ \langle \boldsymbol x, \boldsymbol x' \rangle &= \theta\circ \langle \pi_2,\pi_1 \rangle\circ \langle \boldsymbol x', \boldsymbol x \rangle\\
&= i\circ \theta \circ \langle \boldsymbol x', \boldsymbol x \rangle\;.
\end{align*}

2. Immediate because $\langle\theta,\pi_2\rangle$ is a map into $G_1\otimes_{G_0} X$.

3. Since $\langle\theta,\pi_2\rangle=\langle\act,\pi_2\rangle^{-1}$, we have
\begin{align*}
\boldsymbol x&=\pi_1\circ\langle\boldsymbol x,\boldsymbol x'\rangle\\
&=\pi_1\circ\langle\act,\pi_2\rangle\circ\langle\theta,\pi_2\rangle\circ\langle\boldsymbol x,\boldsymbol x'\rangle\\
&=\pi_1\circ \langle\act,\pi_2\rangle\circ\langle\theta\circ\langle\boldsymbol x,\boldsymbol x'\rangle,\boldsymbol x'\rangle\\
&=\pi_1\circ \bigl\langle\act\circ\langle\theta\circ\langle\boldsymbol x,\boldsymbol x'\rangle,\boldsymbol x'\rangle,\boldsymbol x'\bigr\rangle\\
&=\act\circ\langle\theta\circ\langle\boldsymbol x,\boldsymbol x'\rangle,
\boldsymbol x'\rangle\;.
\end{align*}

4. The unit laws of $G$ give us the equality
\begin{equation}\label{eq:aupxx}
\act\circ\langle u\circ p\circ\boldsymbol x,\boldsymbol x\rangle=\boldsymbol x\;,
\end{equation}
so we obtain
\begin{align*}
u\circ p\circ \boldsymbol x
&= \pi_1\circ\langle u\circ p\circ \boldsymbol x,\boldsymbol x\rangle\\
&= \pi_1\circ\langle\theta,\pi_2\rangle\circ\langle\act,\pi_2\rangle\circ\langle u\circ p\circ\boldsymbol x,\boldsymbol x\rangle&(\langle\theta,\pi_2\rangle=\langle\act,\pi_2\rangle^{-1})\\
&= \pi_1\circ\langle\theta,\pi_2\rangle\circ\langle\act\circ\langle u\circ p\circ\boldsymbol x,\boldsymbol x\rangle,\boldsymbol x\rangle\\
&= \pi_1\circ\langle\theta,\pi_2\rangle\circ\bigl\langle\boldsymbol x,\boldsymbol x\bigr\rangle&\text{[by \eqref{eq:aupxx}]}\\
&= \pi_1\circ\bigl\langle \theta\circ\langle\boldsymbol x,\boldsymbol x\rangle,\boldsymbol x\bigr\rangle\\
&= \theta\circ\langle\boldsymbol x,\boldsymbol x\rangle\;.
\end{align*}

5. The associativity law and the principality of $X$  give us the following equality:
\begin{equation}\label{eq:mthetaa}
(m\otimes \ident_X)\circ (\ident_{G_1}\otimes \langle \theta,\pi_2 \rangle)\circ \langle \boldsymbol g,\langle \boldsymbol x, \boldsymbol x' \rangle \rangle= \langle \theta,\pi_2 \rangle\circ (\act \otimes \ident_X)\circ \langle \boldsymbol g,\langle \boldsymbol x, \boldsymbol x' \rangle \rangle\;.
\end{equation}
Then,
\begin{align*}
m\circ \langle \boldsymbol g, \theta \circ \langle \boldsymbol x, \boldsymbol x' \rangle\rangle  &= \pi_1\circ \langle m\circ \langle \boldsymbol g, \theta\circ \langle \boldsymbol x, \boldsymbol x' \rangle \rangle, \boldsymbol x' \rangle\\
&= \pi_1\circ (m\otimes \ident_X)\circ \langle \boldsymbol g, \langle \theta\circ \langle \boldsymbol x, \boldsymbol x' \rangle, \boldsymbol x' \rangle \rangle\\
 &= \pi_1\circ (m\otimes \ident_X)\circ (\ident_{G_1}\otimes \langle \theta,\pi_2 \rangle)\circ \langle \boldsymbol g,\langle \boldsymbol x, \boldsymbol x' \rangle \rangle\\
&= \pi_1\circ \langle \theta,\pi_2 \rangle\circ (\act \otimes \ident_X)\circ \langle \boldsymbol g,\langle \boldsymbol x, \boldsymbol x' \rangle \rangle\quad \text{[by \eqref{eq:mthetaa}]}\\
&= \pi_1\circ \langle \theta,\pi_2 \rangle\circ \langle \act\circ \langle \boldsymbol g, \boldsymbol x \rangle, \boldsymbol x' \rangle\\
&= \pi_1\circ \langle \theta\circ \langle \act\circ \langle \boldsymbol g, \boldsymbol x \rangle, \boldsymbol x' \rangle,  \boldsymbol x' \rangle\\
&= \theta\circ \langle \act\circ \langle \boldsymbol g,\boldsymbol x \rangle, \boldsymbol x' \rangle\;. 
\end{align*}
\end{proof}

This lemma provides an example of how topological arguments using points can be translated to localic arguments (\cf~\cite{Vi07}). Such an adaptation is valid due to the simple observation that given any locale maps $f,g:X\to Y$ we have $f=g$ if and only if for all locales $Z$ and all maps $\boldsymbol x:Z\to X$ the equation $f\circ\boldsymbol x= g\circ\boldsymbol x$ holds. In particular, in the case of maps in several variables, say $f,g:X\otimes X'\to Y$, a map $Z\to X\otimes X'$ is the same thing as a pairing map $\langle\boldsymbol x,\boldsymbol x'\rangle$, as in Lemma~\ref{lem:localicreasoningwithglobalpoints}.

\subsection{Principal bibundles}

There are several notions of map from a groupoid $H$ to a groupoid $G$ which are based on bibundles between $G$ and $H$ and generalize continuous functors $\varphi:H\to G$. For instance, Hilsum and Skandalis~\cite{HS87} define a map $\varphi:W\to V$ between the spaces of leaves of two smooth foliated manifolds to be a \emph{principal $G$-$H$-bibundle} --- that is, a principal $G$-bundle over $H_0$ --- where $G$ and $H$ are the holonomy groupoids of $V$ and $W$, respectively. In the present paper, following~\cite{Mr99}, by a \emph{Hilsum--Skandalis map} from an \'etale groupoid $H$ to an \'etale groupoid $G$ will be meant the isomorphism class of a principal $G$-$H$-bibundle. Another name for such maps is \emph{bibundle functors} \cite{MeyerZhu}. The \emph{category $\HSG$ of Hilsum--Skandalis maps} is that whose objects are \'etale groupoids and whose morphisms are the Hilsum--Skandalis maps, with the composition in $\HSG$ being induced by the tensor product of principal bibundles.
 
In order to see that the composition in $\HSG$ is well defined let $G$, $H$, and $K$ be \'etale groupoids and $X$ and $Y$ a principal $G$-$H$-bibundle and an $H$-$K$-bibundle, respectively. We need to prove that $X\otimes_H Y$ is a principal $G$-$K$-bibundle; that is, we need to find a suitable isomorphism
\[
\kappa: (X\otimes_H Y)\otimes_{K_0}(X\otimes_H Y)\to G_1\otimes_{G_0} (X\otimes_H Y)\;.
\]
Let $\act$ and $\bct$ be the actions $G_1\otimes_{G_0} X\to X$ and $H_1\otimes_{H_0}Y\to Y$, respectively, and let $\theta:X\otimes_{H_0}X\to G_1$ and $\theta':Y\otimes_{K_0} Y\to H_1$ be the maps such that $\langle\theta,\pi_2\rangle$ and $\langle\theta',\pi_2\rangle$ are the inverses of $\langle\act,\pi_2\rangle$ and $\langle\bct,\pi_2\rangle$, respectively [\cf\ \eqref{eq:theta}]:

\[\begin{array}{c}
\xymatrix{
X\otimes_{H_0} X\ar@<1.0ex>[rr]_{\cong}^{ \langle \theta,\pi_2 \rangle}&& G_1\otimes_{G_0} X \ar@<1.0ex>[ll]^{ \langle \act,\pi_2 \rangle}\;,
}\\
\xymatrix{
Y\otimes_{K_0} Y\ar@<1.0ex>[rr]_{\cong}^{ \langle \theta',\pi_2 \rangle}&& H_1\otimes_{H_0} Y \ar@<1.0ex>[ll]^{ \langle \act,\pi_2 \rangle}\;.
}
\end{array}
\]
For topological groupoids and continuous actions the tensor product isomorphism
\[
\kappa: (X\otimes_H Y)\times_{K_0}(X\otimes_H Y)\to G_1\times_{G_0} (X\otimes_H Y)\;.
\]
can be described directly by the following formula for arbitrary points $x,x'\in X$ and $y,y'\in Y$~\cite{Mr99}:
\[
\kappa(x\otimes y,x'\otimes y') = \bigl(\theta(x\theta'(y,y'),x'),x'\otimes y'\bigr)\;,
\]
where $x\theta'(y,y'):=\bct(x,\theta'(y,y'))$, and $x\otimes y$ and $x'\otimes y'$ are the orbits of $(x,y)$ and $(x',y')$, respectively, in $X\otimes_H Y$.
(It is clear that $\kappa$ is well defined for we have
$\kappa\bigl((x h)\otimes y,x'\otimes y'\bigr)=\kappa\bigl(x\otimes (h y),x'\otimes y'\bigr)$, which is proved using the obvious equality $\theta'(h y,y')=h\theta'(y,y')$ --- \cf\ Lemma~\ref{lem:localicreasoningwithglobalpoints}\eqref{lem:localicreasoningwithglobalpoints5} ---, etc.)
The proof that this is the inverse of $\langle \act',\pi_2\rangle$, where $\act'$ is the action $G_1\times_{G_0} (X\otimes_{H} Y)\to X\otimes_{H} Y$, which for all $g\in G_1$, $x\in X$ and $y\in Y$ is given by
\[
g (x\otimes y) = (g x)\otimes y\;,
\]
is based on simple properties that carry over to localic groupoids in the manner of Lemma~\ref{lem:localicreasoningwithglobalpoints}. This adaptation is straightforward but tedious, and we omit it.

\subsection{Functors}

We provide a localic version of the functor $\langle - \rangle: \grpd\to \HSG$ from the category of \'etale groupoids $\grpd$ into the category of Hilsum--Skandalis maps $\HSG$, presented for topological and Lie groupoids in \cite{Mr99, MeyerZhu,MrcunPhD}.

Let $G$ and $H$ be \'etale groupoids and let $\phi=(\phi_0,\phi_1)$ be a functor of groupoids from $H$ to $G$. Let $\langle\phi\rangle:=G_1\otimes_{G_0} H_0$ be the pullback of $r$ and $\phi_0$:
\[
\xymatrix{
G_1\otimes_{G_0} H_0\ar[rr]^-{\pi_2}\ar[d]_{\pi_1}&&H_0\ar[d]^{\phi_0}\\
G_1\ar[rr]_r&&G_0
}
\]
Then $\langle\phi\rangle$
is a left $G$-locale with projection map $p:=d\circ \pi_1$ and action given by
\begin{equation}
\xymatrix{
G_1 {\otimes_{G_0}} (G_1\otimes_{G_0} H_0)\ar[rrrr]^-{\bigl\langle m\circ(\ident_{G_1}\otimes(\pi_1\circ\pi_2)),\, \pi_2\circ\pi_2\bigr\rangle} &&&& G_1\otimes_{G_0} H_0\;,
}
\end{equation}
or, equivalently,
\begin{equation}
\xymatrix{
G_1 {\otimes_{G_0}} (G_1\otimes_{G_0} H_0)\cong G_2\otimes_{G_0} H_0\ar[rr]^-{m\otimes \ident_{H_0}} && G_1\otimes_{G_0} H_0\;.
}
\end{equation}
And $\langle\phi\rangle$ is also a right $H$-locale with projection the \'etale map $q:=\pi_2$ and action given by:
\begin{equation}
\xymatrix{
(G_1\otimes_{G_0} H_0)\otimes_{H_0} H_1\ [\cong G_1\otimes_{G_0} H_1]\ar[rrrr]^-{\bigl\langle m\circ((\pi_1\circ\pi_1)\otimes\phi_1),\,r\circ\pi_2\bigr\rangle} &&&& G_1\otimes_{G_0} H_0\;.
}
\end{equation}
It is clear that $\langle\phi\rangle$ is a $G$-$H$-bibundle. Moreover, it is a principal $G$-$H$-bibundle:
\begin{align*}
G_1\otimes_{G_0}(G_1\otimes_{G_0} H_0) &\cong (G_1\otimes_{G_0}G_1)\otimes_{G_0} H_0\\
&\cong G_1\otimes_{G_0} (H_0\otimes_{H_0} G_1)\otimes_{G_0} H_0\\
&\cong (G_1\otimes_{G_0} H_0)\otimes_{H_0} (G_1\otimes_{G_0} H_0)\;.
\end{align*} 

\begin{lemma}\label{lem:globalsection}
Let $G$ and $H$ be \'etale groupoids and $(X,p,q)$ be a principal $G$-$H$-bibundle. Then $X\cong \langle\phi\rangle$ as $G$-$H$-bibundles for some  localic functor $\phi:H\to G$ if and only if $q:X\to H_0$ has a global section.    
\end{lemma}

\begin{proof}
Suppose $(X,p,q)$ is a principal $G$-$H$-bibundle and $\psi:X\to G_1\otimes_{G_0} H_0$ is an isomorphism of $G$-$H$-bibundles, for some localic functor $(\phi_0,\phi_1)$ from $H$ to $G$. Notice that $q=\pi_2\circ \psi$ and let us define $s$ by $\psi^{-1}\circ \langle u\circ \phi_0,\ident_{H_0} \rangle $. Then it is clear that $s$ is a global section of $q$. Conversely, suppose that there exists a global section $s:H_0\to X$ of $q$. Let us define $\phi_0=p\circ s$, which is a map of locales from $H_0$ to $G_0$, and $\phi_1=\theta \circ \langle s\circ r, s\circ d \rangle$ (where $\theta=\pi_1\circ \langle \theta,\pi_2\rangle$ and $\langle \theta,\pi_2\rangle=\langle \act,\pi_2\rangle^{-1}$ exists by principality of $X$), which defines a map of locales from $H_1$ to $G_1$. Furthermore, we have the following two properties:

1. $\phi=(\phi_1,\phi_0)$ is a functor --- this corresponds to satisfying the conditions $\phi_1\circ i = i \circ \phi_1$, $\phi_1\circ u=u\circ\phi_1$, $d\circ \phi_1=\phi_0\circ d$, and $(m\circ (\phi_1\otimes\phi_1)\leq \phi\circ m)$, whose adaptation from topological groupoids to localic groupoids is straightforward using properties such as those of Lemma~\ref{lem:localicreasoningwithglobalpoints}.

2. $X\cong \langle \phi \rangle$ --- using again the properties of Lemma \ref{lem:localicreasoningwithglobalpoints} it is routine to verify that the maps of locales $\psi: \langle \phi \rangle\to X$ given by $\psi\circ \langle \boldsymbol g, \boldsymbol b \rangle:= \act\circ \langle  \boldsymbol g, s \circ  \boldsymbol b  \rangle$ for all ``points'' $\boldsymbol g$ and $\boldsymbol x$, and $\varphi: X\to \langle \phi \rangle$ given by $\varphi\circ \boldsymbol x:= \langle \theta\circ \langle \boldsymbol x, s\circ q\circ \boldsymbol x \rangle, q\circ \boldsymbol x \rangle$ for all ``points'' $\boldsymbol x$, are inverse to each other.
\end{proof}

\subsection{Morita equivalence}

In the literature there are several equivalent definitions of Morita equivalence for groupoids of various kinds, such as localic groupoids, topological groupoids, or Lie groupoids. For discrete groupoids the natural notion of equivalence is that of an equivalence of categories, which can be formulated in two ways~\cite{maclane}: either as (1)~a full and faithful functor which is surjective on objects up to isomorphisms; or as (2)~a pair of functors going in opposite directions together with a pair of natural isomorphisms defining an adjunction which is both a reflection and a coreflection. For discrete groupoids these definitions are equivalent (if we assume the axiom of choice), but this is not the case for more general groupoids, for which (2)~is in general too strong. But (1)~can be phrased for arbitrary groupoids, in particular for localic groupoids as we now describe: a morphism (localic functor) $\phi:G\to H$ of \'etale groupoids is an \emph{essential equivalence} if both of the following conditions hold (\cf\ \cite{MrcunPhD}*{pp.\ 18--19} or \emph{weak equivalence} in \cite{elephant}*{Lemma 5.3.12}).
\begin{enumerate}
\item
$
d\circ\pi_1: H_1\otimes_{H_0} G_0\to H_0
$
is an open surjection (where $H_1\otimes_{H_0} G_0$ is the pullback of $r$ and $\phi_0$);
\item The following is a pullback diagram:
\[
\xymatrix{
G_1\ar[rr]^{\phi_1}\ar[d]_{\langle d,r\rangle}&&H_1\ar[d]^{\langle d,r\rangle}\\
G_0\otimes G_0\ar[rr]_{\phi_0\otimes \phi_0}&&H_0\otimes H_0
}
\]
\end{enumerate}
The notion of \emph{Morita equivalence} of localic \'etale groupoids is then defined to be the least equivalence relation that contains the relation of being related by an essential equivalence, and it can be shown that $G$ and $H$ are Morita equivalent if and only if there exists a groupoid $K$ and a span of essential equivalences as follows:
\[
\xymatrix{
&&K\ar[dll]_\phi\ar[drr]^\psi\\
G&&&&H
}
\]
In~\cite{MrcunPhD} it is proved that, for topological groupoids (\cf~\cite{Mr99} for Lie groupoids), two \'etale groupoids $G$ and $H$ are Morita equivalent if and only if they are isomorphic in the category of Hilsum--Skandalis maps; that is, if and only if there exists a principal $G$-$H$-bibundle $X$ and a principal $H$-$G$-bibundle $Y$, and isomorphisms $\xi$ and $\zeta$ of $G$-$G$-bibundles and $H$-$H$-bibundles, respectively, as follows:
\[ \xi: X\otimes_H Y\stackrel\cong\longrightarrow G\quad\ \text{and}\quad\ \zeta:Y\otimes_G X\stackrel\cong\longrightarrow H\;.     \]
The same proof carries through to localic groupoids. Moreover, we highlight the following fact~\cite{MrcunPhD}*{Prop.\ II.1.6}: if $\phi:G\to H$ is an essential equivalence then $\langle\phi\rangle$ is an isomorphism in $\HSG$. Although not stated as such, the proof of this result relies on showing that the dual $\langle\phi\rangle^*$ is an inverse of $\langle\phi\rangle$. Hence we obtain the following result, already carried over to our setting of localic groupoids:

\begin{lemma}\label{lem:moritaequivdefs}
Let $G$ and $H$ be \'etale groupoids. The following conditions are equivalent:
\begin{enumerate}
\item\label{lem:moritaequivdefs1} $G$ and $H$ are Morita equivalent;
\item\label{lem:moritaequivdefs2} $G$ and $H$ are isomorphic in $\HSG$;
\item\label{lem:moritaequivdefs3} $G$ and $H$ are \emph{unitarily isomorphic} in $\HSG$; that is, there exists a principal $G$-$H$-bibundle $X$ with inverse in $\HSG$ given by $X^*$.
\end{enumerate}
\end{lemma}

\begin{proof}
The equivalence $\eqref{lem:moritaequivdefs1}\Leftrightarrow\eqref{lem:moritaequivdefs2}$ has already been mentioned above. $\eqref{lem:moritaequivdefs3}\Rightarrow\eqref{lem:moritaequivdefs2}$ is trivial, and for $\eqref{lem:moritaequivdefs1}\Rightarrow\eqref{lem:moritaequivdefs3}$ assume that $G$ and $H$ are Morita equivalent and that there are essential equivalences $\phi$ and $\psi$ as follows:
\[
\xymatrix{
G&&\ar[ll]_\phi \ar[rr]^\psi K&& H
}
\]
Then, taking $X=\langle\phi\rangle\otimes_K\langle\psi\rangle^*$, we obtain an invertible principal $G$-$H$-bibundle $X$ whose inverse is $X^*$.
\end{proof}

\section{Hilsum--Skandalis sheaves}\label{sec:principalsheaves}

In this section we achieve the main goal of this paper, which is to study principal bundles, Hilsum--Skandalis maps, and Morita equivalence for inverse quantal frames. In doing so we shall restrict to groupoid sheaves rather than arbitrary actions, which for Morita equivalence (of \'etale groupoids) is enough, due to Lemma~\ref{funct2, lemma: covering}.

\subsection{Principal sections and free actions}

Let $Q$ be an inverse quantal frame and $X$ a $Q$-sheaf. The formula for the inner product of $X$ obtained in~\cite{GSQS} is recalled in \eqref{b_!} (\cf\ Theorem~\ref{innerproduct2}). In plain english, the formula states that the inner product of two local sections $s,t\in\sections_X$ is the largest element of $Q$ that ``translates'' $t$ to a subsection of $s$ by acting on $t$ the left. Hence, one expects the $Q$-action on $X$ to be free if and only if for all local sections $s$ the inner product $\langle s,s\rangle$ is an element below $e\in Q$.  A local section $s$ that possesses this property will be referred to as a \emph{principal section}, and it can be given the following four equivalence definitions:

\begin{lemma}\label{principalsections}
Let $Q$ be an inverse quantal frame, $X$ a $Q$-sheaf, and $s\in\sections_X$. The following conditions are equivalent:
\begin{enumerate}
\item\label{free1} $\langle s,s\rangle\in Q_0$;
\item\label{free2} $\spp_X(s)=\langle s,s\rangle$;
\item\label{free3pre} For all $a\in Q$ such that $as=s$ we have
$a\spp_X(s)=\spp_X(s)$.
\item\label{free3} For all $a\in Q$ such that $\spp_Q(a^*)=\spp_X(s)$ and $as=s$ we have
$a\in Q_0$.
\end{enumerate}
\end{lemma}

\begin{proof}
\eqref{free1}$\Rightarrow$\eqref{free2}. If $\langle s,s\rangle\in Q_0$ then, by \eqref{supportformulamod}, we have $\spp_X(s)=\langle s,s\rangle\wedge e=\langle s,s\rangle$.

\eqref{free2}$\Rightarrow$\eqref{free3pre}. Assume \eqref{free2} and let $a\in Q$ be such that $as=s$. Then
\[
\spp_X(s)={\langle s,s\rangle} = {\langle as,s \rangle} = a{\langle s,s\rangle} = a\spp_X(s)\;.
\]

\eqref{free3pre}$\Rightarrow$\eqref{free3}. This is immediate because $a\spp_Q(a^*)=a$.

\eqref{free3}$\Rightarrow$\eqref{free1}.
The inequality
$\langle s,s\rangle s\le s$ gives us
\[
(\langle s,s\rangle \vee\spp_X(s))s=\langle s,s\rangle s\vee s=s\;,
\]
where $\spp_Q((\langle s,s\rangle \vee\spp_X(s))^*)=\spp_X(s)$ because $\spp_Q(\langle s,s\rangle)=\spp_X(s)$ [by \eqref{supportformulamod}], so using \eqref{free3} we conclude
$\langle s,s\rangle\vee\spp_X(s)\in Q_0$, whence $\langle s,s\rangle\in Q_0$.
\end{proof}

From here on the set of principal sections of a $Q$-sheaf $X$ will be denoted by $\psections_X$. Principal sections are closely related to $Q_\ipi$-valued inner products of local sections, as the following two lemmas show.

\begin{lemma}\label{pairsofprincipal}
Let $Q$ be an inverse quantal frame and $X$ a $Q$-sheaf. Let also $s,t\in\sections_X$ and $u\in Q_\ipi$.
\begin{enumerate}
\item\label{pairsofprincipal1} If $s,t\in\psections_X$ then $\langle s,t\rangle\in Q_\ipi$.
\item\label{pairsofprincipal2} If $t\in\psections_X$, $s=ut$ and $\spp_X(t)=\spp_Q(u^*)$ then $s\in\psections_X$, $\langle s,t\rangle=u$ and $\spp_X(s)=\spp_Q(u)$.
\item\label{pairsofprincipal3} If $\langle s,t\rangle=u$, $\spp_X(s)=\spp_Q(u)$ and $\spp_X(t)=\spp_Q(u^*)$ then $s,t\in\psections_X$ and both $s=ut$ and $t=u^*s$.
\end{enumerate}
\end{lemma}

\begin{proof}
We begin with the proof of \eqref{pairsofprincipal1}.
Let $s,t\in\psections_X$. The hypothesis $s\in\psections_X$ gives us the condition $\langle s,t\rangle\langle s,t\rangle^*\le e$:
\[
\langle s,t\rangle\langle s,t\rangle^*=\langle s,t\rangle\langle t,s\rangle
\le\langle s,s\rangle\le e\;.
\]
Similarly, the hypothesis $t\in\psections_X$ gives us the condition
$\langle s,t\rangle^*\langle s,t\rangle\le e$, and thus $\langle s,t\rangle\in Q_\ipi$.

Now the proof of \eqref{pairsofprincipal2}. With the given hypotheses we have
\[
\langle s,t\rangle = \langle ut,t\rangle=u\langle t,t\rangle=u\spp_X(t)=u\spp_Q(u^*)=u\;,
\]
so $\langle s,t\rangle=u$. And
\[
\langle s,s\rangle = \langle ut,ut\rangle = u\langle t,t\rangle u^*=u\spp_X(t) u^*=u\spp_Q(u^*)u^*=uu^*=\spp_Q(u)\;,
\]
so $s\in \psections_X$ and $\spp_X(s)=\spp_Q(u)$.

Finally we prove \eqref{pairsofprincipal3}. With the given hypotheses we have $ut=\langle s,t\rangle t\le s$ because $t$ is a Hilbert section. Hence, since $s$ is a local section, we obtain $ut=s$:
\begin{align*}
ut &= \spp_X(ut) s\\
&=\spp_Q(u \spp_X(t)) s \quad\quad\text{($\spp_X$ is stable)}\\
&=\spp_Q(u\spp_Q(u^*)) s\\
& = \spp_Q(u)s\\
& = \spp_X(s)s=s\;.
\end{align*}
We also obtain $t=\spp_X(t)t=\spp_Q(u^*)t=u^*ut=u^*s$, and
\begin{align*}
\langle s,s\rangle &= \langle ut,s\rangle = u\langle t,s\rangle =uu^*\le e\;,\\
\langle t,t\rangle &= \langle t,u^*s\rangle=\langle t,s\rangle u=u^*u\le e\;,
\end{align*}
so both $s$ and $t$ are principal sections.
\end{proof}

Let $Q$ be an inverse quantal frame and $X$ a $Q$-sheaf. A Hilbert basis $\sections$ of $X$ will be said to be \emph{principal} if $\sections\subset\psections_X$.

\begin{lemma}\label{princbasis}
Let $Q$ be an inverse quantal frame, $X$ a $Q$-sheaf and $\sections\subset X$ a Hilbert basis. The following conditions are equivalent:
\begin{enumerate}
\item\label{princbasis1} $\sections$ is principal;
\item\label{princbasis2} $\langle s,t\rangle\in Q_\ipi$ for all $s,t\in\sections$.
\end{enumerate}
\end{lemma}

\begin{proof}
The implication \eqref{princbasis1}$\Rightarrow$\eqref{princbasis2} follows immediately from Lemma~\ref{pairsofprincipal}. In order to prove
\eqref{princbasis2}$\Rightarrow$\eqref{princbasis1} assume that $\langle s,t\rangle\in Q_\ipi$ for all $s,t\in\sections$, and let $s\in\sections$. Using Parseval's identity we obtain
\[
\langle s,s\rangle = \V_{t\in\sections} \langle s,t\rangle\langle t,s\rangle
=\V_{t\in \sections}\langle s,t\rangle\langle s,t\rangle^*\le e\;,
\]
and thus $\langle s,s\rangle\in Q_0$, so $s\in\psections_X$.
\end{proof}

Let $Q$ be an inverse quantal frame and $X$ a $Q$-sheaf.
We note that the following conditions are equivalent:
\begin{enumerate}
\item A principal Hilbert basis of $X$ exists;
\item $\psections_X$ is a Hilbert basis of $X$;
\item $\V\psections_X=1_X$ (\cf\ Lemma~\ref{subhilbertbasis}).
\end{enumerate}
If these equivalent conditions hold we say that $X$ is a \emph{principally covered} $Q$-sheaf. The following fact confirms that the intuition behind the definition of principal section (\cf\ text immediately preceding Lemma~\ref{principalsections}) is correct:

\begin{theorem}\label{thm:freeness}
Let $G$ be an \'etale groupoid with quantale $Q:=\opens(G)$ and $X$ a $G$-sheaf. If $X$ is a principally covered $Q$-sheaf the action of $G$ on $X$ is free; that is,
\[
\langle \act,\pi_2\rangle: G_1\otimes_{G_0} X\to X\otimes_{X/G} X
\]
is a regular monomorphism in $\Loc$.
\end{theorem}

\begin{proof}
Let $X$ be be a free $Q$-sheaf with $G$-action $\act$, and consider an element $u\otimes t\in Q\otimes_{Q_0} X$ where $u\in Q_\ipi$ and $t\in\psections_X$. Without loss of generality we shall assume that $\spp_Q(u^*)=\spp_X(t)$. Let $s=ut$. Then
\begin{align*}
[ \act^*,\pi_2^*](s\otimes t) &= \V_{v\in Q_\ipi} v\otimes(v^* s\wedge t) & \text{[by \eqref{eq:rightadjformula2}]}\\
&= \V_{v\in Q_\ipi} v\otimes\spp_X(v^* s\wedge t)t & \text{($t$ is a local section)}\\
&=\V_{v\in Q_\ipi} v\spp_X(v^* s\wedge t)\otimes t\\
&=\langle s,t\rangle\otimes t &\text{(by Theorem~\ref{innerproduct2})}\\
&=u\otimes t&\text{[by Lemma~\ref{pairsofprincipal}\eqref{pairsofprincipal2}]}\;.
\end{align*}
Let $\xi$ be an arbitrary element of $Q\otimes_{Q_0} X$. This is necessarily of the form $\V_i u_i\otimes t_i$ with $u_i\in Q_\ipi$, $t_i\in\psections_X$ and $\spp_Q(u_i^*)=\spp_X(t_i)$ for all $i$, and thus for each $i$ we have
$u_i\otimes t_i = [\act^*,\pi_2^*](u_i t_i\otimes t_i)$. Hence,
\[
\xi=[\act^*,\pi_2^*]\bigl(\V_i u_i t_i\otimes t_i\bigr)\;,
\]
so $\langle\act,\pi_2\rangle^*$ is a surjective homomorphism of locales.
\end{proof}

\subsection{Bisheaves and transitive actions}\label{sec:transitivity}

Let $G$ be an \'etale groupoid. By a \emph{$G$-bisheaf} will be meant a $G$-sheaf $X$ for which the projection to the orbit locale $\pi:X\to X/G$ is a local homeomorphism. Analogously, for an inverse quantal frame $Q$ we define a \emph{$Q$-bisheaf} to be a $Q$-sheaf $X$ whose right $I(X)$-locale structure makes it a (right) $I(X)$-sheaf. Clearly, $X$ is a $G$-bisheaf if and only if it is an $\opens(G)$-bisheaf. Given a $Q$-bisheaf $X$ we shall write $\spp_X:X\to Q_0$ for the usual support of $X$, and $\tspp:X\to I(X)$ for the $I(X)$-valued support, which for all $x\in X$ is given by
\[
\tspp(x) = 1_Qx
\]
(\cf\ Example~\ref{exm:openIXlocale}).
As usual the inner product of the $Q$-sheaf structure is denoted by
\[
\langle-,-\rangle:X\times X\to Q\;,
\]
and the $I(X)$-sheaf inner product of is denoted by
\[
[-,-]:X\times X\to I(X)\;.
\]
Note that for all $x,y\in X$ the latter is defined by
\[
[x,y]=\tspp(x\wedge y)=1_Q(x\wedge y)\;.
\]
We shall write $\sections_X$ for the set of local sections of $X$ regarded as a $Q$-sheaf, and $\tsections_X$ for the set of local sections of $X$ regarded as an $I(X)$-sheaf. The intersection of these sets is referred to as the set of \emph{local bisections} of $X$ and we denote it by
\[
\bisections_X:=\sections_X\cap\tsections_X\;.
\]

\begin{lemma}\label{lem:bisheafbasis}
Let $Q$ be an inverse quantal frame and $X$ a $Q$-bisheaf. Then
$\bisections_X$ is a Hilbert basis with respect to both sheaf structures of $X$.
\end{lemma}

\begin{proof}
Every element $x\in X$ can be written as $x=\V_i s_i$ for some family $(s_i)$ in $\sections_X$. This is true, in particular, for any $s\in \tsections_X$, in which case the elements $s_i$ are in $\bisections_X$, and so $\bisections_X$ is join-dense in $X$. The conclusion that it is a Hilbert basis (with respect to both inner products) follows from Lemma~\ref{subhilbertbasis}.
\end{proof}

For a topological groupoid $G$ and a continuous $G$-action on a topological space $X$ the action of $G$ is tautologically transitive over the orbit space; that is, for all $x,y\in X$ in the same orbit there is (by definition of orbit) a groupoid arrow $g$ such that $gy=x$. This is equivalent to saying that the pairing map $\langle\act,\pi_2\rangle:G_1\times_{G_0} X\to X\times_{X/G} X$ is surjective. The following result shows that something analogous happens for localic \'etale groupoids and bisheaves:

\begin{theorem}\label{thm:transitivity}
Let $G$ be an \'etale groupoid and $X\equiv (X,p,\act)$ a $G$-bisheaf. Then the pairing map
\[
\langle\act,\pi_2\rangle:G_1\otimes_{G_0} X\to X\otimes_{X/G} X
\]
is an epimorphism in $\Loc$.
\end{theorem}

\begin{proof}
For each $u\in Q_\ipi$ and $t\in\bisections_X$ define a mapping
\[
\varphi:Q\times X\to X\otimes_{I(X)} X
\]
by
\[
\varphi(u, t) = ut\otimes t\;.
\]
This obviously preserves joins of partial units in the left variable, and in order to see that it also preserves joins of local bisections in the right variable let $t:=\V_{i\in I} t_i$ be such a join. Then, taking into account that for all $i,j\in I$ we have $t_j\tspp(t_i)=t\tspp(t_j)\tspp(t_i)=t\tspp(t_i)\tspp(t_j)=t_i\tspp(t_j)$ and
\[
t_i\otimes t_j=t_i\tspp(t_i)\otimes t_j= t_i\otimes t_j\tspp(t_i)=t_i\otimes t_i\tspp(t_j)\le t_i\otimes t_i\;,
\]
we obtain
\[
\varphi(u,\V_i t_i) = \V_{ij} ut_i\otimes t_j=\V_i ut_i\otimes t_i = \V_i\varphi(u,t_i)\;.
\]
Moreover, for all $b\in Q_0$ we have
\[
\varphi(ub,t)=ubt\otimes t=ubt\tspp(bt)\otimes t=ubt\otimes t\tspp(bt)=ubt\otimes bt=\varphi(u,bt)\;,
\]
and thus $\varphi$ factors uniquely through a sup-lattice homomorphism
\[
\varphi^\sharp: Q\otimes_{Q_0} X\to X\otimes_{I(X)} X
\]
satisfying $\varphi^\sharp(u\otimes t)=ut\otimes t$. 
For all $s,t\in \bisections_X$ we have
\begin{align*}
\varphi^\sharp\bigl([\act^*,\pi^*_2](s\otimes t)\bigr) &= \varphi^\sharp(\act^*(s)\wedge \pi^*_2(t))\\
&= \varphi^\sharp\bigl((\V_{u\in Q_\ipi} u\otimes u^*s) \wedge (1_Q\otimes t)\bigr)\\
&= \varphi^\sharp\bigl(\V_{u\in Q_\ipi} u\otimes (u^*s\wedge t)\bigr)\\
&= \varphi^\sharp\bigl(\V_{u\in Q_\ipi} u\otimes \spp_X(u^*s\wedge t)t\bigr)\quad\quad(t\in\sections_X)\\
&= \varphi^\sharp\bigl(\V_{u\in Q_\ipi} u\spp_X(u^*s\wedge t)\otimes t\bigr)\quad\quad\text{($\otimes$ is over $G_0$)}\\
&= \V_{u\in Q_\ipi} u\spp_X(u^*s\wedge t)t\otimes t= \V_{u\in Q_\ipi} u(u^*s\wedge t)\otimes t\\
&=\V_{u\in Q_\ipi} (s\wedge ut)\otimes t= (s\wedge 1_Qt)\otimes t\\
&= s\otimes (1_Qt\wedge t)\quad\quad\text{(now $\otimes$ is over $I(X)$)}\\
&= s\otimes t\;.
\end{align*}
Therefore $\varphi^\sharp\circ\langle\act,\pi_2\rangle^*=\ident$, so $\langle\act,\pi_2\rangle$ is an epimorphism in $\Loc$.
\end{proof}

The following result will not be needed, but most of it follows from the proof of the above theorem and it is interesting in its own right.

\begin{corollary}
Let $G$ be an \'etale groupoid and $X\equiv (X,p,\act)$ a $G$-bisheaf. Then the pairing map
$
\langle\act,\pi_2\rangle:G_1\otimes_{G_0} X\to X\otimes_{X/G} X
$
is semiopen, and its direct image $\langle\act,\pi_2\rangle_!$ satisfies, for all $u\in Q_\ipi$ and $t\in\bisections_X$,
\begin{equation}\label{cor:directimage}
\langle\act,\pi_2\rangle_!(u\otimes t) = ut\otimes t\;.
\end{equation}
\end{corollary}

\begin{proof}
Let $\varphi^\sharp$ be defined as in the proof of Theorem~\ref{thm:transitivity}, where it has been seen that $\varphi^\sharp$ splits $\langle\act,\pi_2\rangle^*$. In order to prove that $\langle\act,\pi_2\rangle$ is semiopen we show that $\varphi^\sharp$ is left adjoint to $\langle\act,\pi_2\rangle^*$ by proving the adjunction unit (the counit is given by the splitting). Let $u\in Q_\ipi$ and $t\in \bisections_X$. The following derivation shows that $\langle\act,\pi_2\rangle^*\circ\varphi^\sharp\ge\ident$:
\begin{align*}
[\act^*,\pi^*_2]\bigl(\varphi^\sharp(u\otimes t)\bigr) &= [\act^*,\pi^*_2](ut\otimes t)= \act^*(ut)\wedge \pi^*_2(t)\\
&=\V_{v\in Q_\ipi} v\otimes (v^*ut\wedge t)=\V_{v\in Q_\ipi}  v\otimes \spp_X(v^*ut\wedge t)t\\
&=\V_{v\in Q_\ipi}  v\spp_X(v^*ut\wedge t)\otimes t\quad\quad\text{($\otimes$ is over $G_0$)}\\
&= \langle ut, t \rangle\otimes t\quad\quad \text{(by Theorem~\ref{innerproduct2})} \\
&= u\langle t, t \rangle\otimes t\\
&\ge u\spp_X(t)\otimes t\\
&=u\otimes \spp_X(t)t= u\otimes t\;.
\end{align*}
This shows that $\langle\act,\pi_2\rangle_!=\varphi^\sharp$ and proves \eqref{cor:directimage} because this formula is satisfied by $\varphi^\sharp$.
\end{proof}

\subsection{Principal sheaves}

Let $G$ be an \'etale groupoid. By a \emph{principal $G$-sheaf} will be meant a principal $G$-bundle (over $X/G$) which is also a $G$-sheaf. Such principal bundles are also known as \emph{covering bibundle functors} \cite{MeyerZhu}. Now we shall find a quantale-theoretic characterization of this notion. Naturally, we expect principal $G$-sheaves to correspond to principally covered $\opens(G)$-bisheaves, since, due to Theorem~\ref{thm:freeness} and Theorem~\ref{thm:transitivity}, the latter correspond to free and transitive $G$-actions. That this intuition is correct will be confirmed below in Theorem~\ref{thm:prQsh} after two lemmas.

\begin{lemma}\label{funct2, lemma: b_!=inner}
Let $G$ be an \'etale groupoid with quantale $Q:=\opens(G)$ and let $X$ be a principal $G$-sheaf with $Q$-valued inner product $\langle-,-\rangle$. Then the map $\theta:X\otimes_{X/G} X\to G_1$ is a local homeomorphism whose direct image is given for all $x,y\in X$ by
\[
\theta_!(x\otimes y)={\langle x,y \rangle}\;. 
\]
Furthermore, if the left projection $p:X\to G_0$ is a surjection then so is $\theta$.
\end{lemma}

\begin{proof}
The projection $p$ is a local homeomorphism by hypothesis, and the map $\pi_1:G_1\otimes_{G_0} X\to G_1$ is a local homeomorphism because it is a pullback of $p$ along $d$ [\cf\ \eqref{pi1pbalongd}]. Therefore $\theta$ is a local homeomorphism because it is the composition of $\pi_1$ with the isomorphism $\langle\theta,\pi_2\rangle$:
\[
\xymatrix{
G_1\otimes_{G_0} X \ar@{<-}[rrr]^-{\langle \theta, \pi_2\rangle}_-{\cong}\ar[drr]_{\pi_1}  &&& X\otimes_{X/G} X\ar[ld]^-{\theta}   \\
&& G_1
}
\]
It also follows that $\theta$ is a surjection if $p$ is, for in this case $\pi_1:G_1\otimes_{G_0} X\to G_1$ is a surjection, due to Corollary~\ref{lem:surjstab}.
Next, since $[\act^*,\pi_2^*]=\langle \act,\pi_2\rangle^*=\langle\theta,\pi_2\rangle_!$, we obtain the following commutative diagram of sup-lattice homomorphisms:
\[
\xymatrix{
Q\otimes_{Q_0} X \ar[drr]_{(\pi_1)_!}  &&& X\otimes_{I(X)} X\ar[ld]^-{\theta_!} \ar[lll]_-{[\act^*,\pi^*_2]}^-\cong   \\
&& Q
}
\]
Taking into account that $(\pi_1)_!(q\otimes x)=q\spp_X(x)$ [see Lemma~\ref{directimage} and \eqref{directimagepi2}], and also the formula \eqref{eq:rightadjformula1} for $\act^*$, for all $x,y\in X$ we have 
\begin{align*}
\theta_!(x\otimes y) &= \bigl(\pi_1)_!([\act^*,\pi^*_2](x\otimes y)\bigr)= (\pi_1)_!\bigl(\V_{u\in Q_\ipi} u\otimes (u^*x \wedge y)\bigr)\\
&= \V_{u\in Q_\ipi} (\pi_1)_!\bigl(u\otimes (u^*x \wedge y)\bigr)= \V_{u\in Q_\ipi} u\spp_X(u^*x \wedge y)\\
&=\langle x,y\rangle\;,
\end{align*}
where the last step is an application of Theorem~\ref{innerproduct2}. 
\end{proof}

Given a bisheaf $X$, let us again write $\tsections_X$ for the set of local sections of the right projection.

\begin{lemma}\label{lem:pisections}
Let $G$ be an \'etale groupoid with quantale $Q:=\opens(G)$ and $X$ a principal $G$-sheaf with $Q$-valued inner product $\langle-,-\rangle$. For all $s\in\tsections_X$ we have $\langle s,s\rangle\in Q_0$.
\end{lemma}

\begin{proof}
Let us write $\Delta$ for the following pairing map in the category of $I(X)$-locales:
\[
\xymatrix{
&&X\ar[dll]_{\ident}\ar[drr]^{\ident}\ar@{.>}[d]|-{\Delta=\langle\ident,\ident\rangle}\\
X&&X\otimes_{I(X)} X\ar[ll]^{\pi_1}\ar[rr]_{\pi_2}&&X
}
\]
By Lemma~\ref{directimagepairing} the local sections $s\in\tsections_X$ satisfy
\[
\Delta_!(s)=s\otimes s\;.
\]
Due to principality the following diagram commutes [\cf\ Lemma~\ref{lem:localicreasoningwithglobalpoints}\eqref{lem:localicreasoningwithglobalpoints4}]
\[
\xymatrix{
X\otimes_{I(X)} X\ar[rr]^-\theta&&G_1\\
X\ar[u]^\Delta\ar[rr]_p&&G_0\ar[u]_u
}
\]
and thus for all $s\in\tsections_X$ we have, due to Lemma~\ref{funct2, lemma: b_!=inner},
\[
\langle s,s\rangle = \theta_!(s\otimes s) =\theta_!(\Delta_!(s))=u_!(p_!(s))\;.
\]
Hence, $\langle s,s\rangle\in Q_0$.
\end{proof}

Now we arrive at the promised module-theoretic characterization of principal sheaves:

\begin{theorem}\label{thm:prQsh}
Let $G$ be an \'etale groupoid and $X$ a $G$-sheaf. The following conditions are equivalent:
\begin{enumerate}
\item\label{thm:prQsh1} $X$ is a principally covered $\opens(G)$-sheaf and it is a bisheaf;
\item\label{thm:prQsh2} $X$ is a principal $G$-sheaf.
\end{enumerate}
Furthermore, if these equivalent conditions hold, the set of principal sections $\psections_X$ coincides with the set of local bisections $\bisections_X$.
\end{theorem}

\begin{proof}
Let us prove the implication \eqref{thm:prQsh1}$\Rightarrow$\eqref{thm:prQsh2}. Assume that $X$ is a principally covered bisheaf, and let $\act:G_1\otimes_{G_0} X\to X$ be the action. By Theorem~\ref{thm:freeness} the pairing map
$\langle\act,\pi_2\rangle:G_1\otimes_{G_0} X\to X\otimes_{X/G} X$
is a regular monomorphism, so by Theorem~\ref{thm:transitivity} it is an isomorphism.

Now we prove the implication \eqref{thm:prQsh2}$\Rightarrow$\eqref{thm:prQsh1}. Assume that $X$ is a principal $G$-sheaf. Then it is a bisheaf, by Lemma~\ref{funct2, lemma: covering}. Also, Lemma~\ref{lem:pisections} implies that $\bisections_X\subset\psections_X$, and therefore $X$ is principally covered.

In order to conclude the proof we need to prove the inclusion $\psections_X\subset\bisections_X$. We shall do this by showing that $\psections_X$ is a Hilbert basis for the inner product $[-,-]:X\times X\to I(X)$. Let $x\in X$. Since $X$ is principally covered we have $\V_{s\in\psections_X} s=1_X$, and thus
\[
\V_{s\in\psections_X} s[s,x] = \V_{s\in\psections_X} s\tspp(s\wedge x)
=\V_{s\in\psections_X} s\wedge 1_Q(s\wedge x)\ge \V_{s\in\psections_X} s\wedge x=x\;.
\]
So we have $\V_{s\in\psections_X} s[s,x]\ge x$. In order to prove that this is in fact an equality (and thus $\psections_X$ is a Hilbert basis for the $I(X)$-valued inner product), let us show that $s[s,x]\le x$ or, equivalently, $s\wedge 1_Q(s\wedge x)\le x$ for all $s\in \psections_X$. For this it suffices to prove that
\begin{equation}\label{sufficesequality}
s\wedge u(s\wedge x)\le x
\end{equation}
for all $s\in\psections_X$ and $u\in Q_\ipi$. First let us notice that
\[
s\wedge u(s\wedge x)=s\wedge u(s\wedge \spp_X(s)x)\le s\wedge us\wedge u\spp_X(s)x = t\wedge ux'
\]
where $t:=s\wedge us$ and $x':=\spp_X(s)x$. Since $s$ and $us$ are local sections and both $t\le s$ and $t\le us$, the section $t$ equals both $\spp_X(t)s$ and $\spp_X(t)us$, so
$\spp_X(t)us\le s$, and thus
\[
\bigl(\overbrace{\spp_X(t)u\spp_X(s)\vee\spp_X(s)}^{a}\bigr) s = \spp_X(t)us\vee s=s\;.
\]
Notice that $\spp_X(a^*)=\spp_X(s)$. Therefore $a\in Q_0$ because $s$ is principal (\cf\ Lemma~\ref{principalsections}), and it follows that
\[
t\wedge ux'=t\wedge\spp_X(t)u\spp_X(s)x\le t\wedge ax\le ax\le x\;.
\]
This proves \eqref{sufficesequality}, so we conclude that every principal section is a local bisection.
\end{proof}

From here on if $Q$ is an inverse quantal frame we shall refer to a principally covered $Q$-sheaf that is also a $Q$-bisheaf as a \emph{principal $Q$-sheaf}. 
We conclude this section by proving additional properties of principal sheaves. In particular the following one will be needed later on:

\begin{theorem}\label{funct2, theo: principalsheaves2}
Let $Q$ be an inverse quantal frame and $X$ be a principal $Q$-sheaf. For all $s,t\in\bisections_X$ such that $ \tspp(s)=  \tspp(t)$ there is one and only one $u\in Q_\ipi$ such that $s=ut$ with $\spp_Q(u^*)=\spp_X(t)$, and moreover $u=\langle s,t\rangle$.
\end{theorem}

\begin{proof}
Notice that $X\otimes_{I(X)} X$ is an $I(X)$-sheaf for which, by Theorem~\ref{XOX}, the following set is a Hilbert basis:
\[
\sections=\{ s\otimes t\st\ s\in\bisections_X,\, t\in\bisections_X \}\;.
\]
Let $G$ be an \'etale groupoid such that $Q\cong\opens(G)$ and let $\act$ be the groupoid action corresponding to this $Q$-sheaf.
By principality we have an isomorphism
\[
\langle\act,\pi_2\rangle:Q\otimes_{Q_0} X\stackrel{\cong}\longrightarrow X\otimes_{I(X)} X\;,
\]
so $Q\otimes_{Q_0} X$ an $I(X)$-sheaf for which a Hilbert basis is given by $\sections' = [\act^*,\pi^*_2](\sections)$. Hence, we have
\begin{align*}
\sections' 
&= \{ [\act^*,\pi^*_2](s\otimes t) \st s,t\in\bisections_X\}\\
&=\{ \act^*(s)\wedge \pi^*_2(t) \st s,t\in\bisections_X\}\\
&=\bigl\{ \bigl(\V_{v\in Q_\ipi} v\otimes v^*s\bigr) \wedge (1_Q\otimes t) \st s,t\in\bisections_X\bigr\}\\
&=\bigl\{ \V_{v\in Q_\ipi} v\otimes (v^*s \wedge t) \st s,t\in\bisections_X\bigr\}\\
&=\bigl\{ \V_{v\in Q_\ipi} v\otimes \spp_X(v^*s \wedge t)t \st s,t\in\bisections_X\bigr\}& \text{($v^*s\wedge t$ is a subsection of $t$)}\\
&=\bigl\{ \V_{v\in Q_\ipi} v\spp_X(v^*s \wedge t)\otimes t \st s,t\in\bisections_X\bigr\}\\
&=\{ \langle s,t \rangle\otimes t \st s,t\in\bisections_X\}\;.& \text{(by Theorem~\ref{innerproduct2})}
\end{align*}
Recall that, since $X$ is a $G$-locale, the following diagram is commutative:
\[
\vcenter{\xymatrix{
 Q\otimes_{Q_0} X  \ar[d]_{\pi_1} \ar[r]^-{\act} & X \ar[d]^{p}  \\
 Q  \ar[r]_-{d} & Q_0 }
}
\]
Then the principality of $X$ gives us the commutativity of
\[
\vcenter{\xymatrix{
& X\otimes_{I(X)} X \ar[dr]^{\pi_1} \\
 Q\otimes_{Q_0} X \ar[ur]_{\cong}^{\langle \act, \pi_2 \rangle}  \ar[d]_{\pi_1} \ar[rr]^-{\act} && X \ar[d]^{p}  \\
 Q  \ar[rr]_-{d} && Q_0 }
}
\]
and thus we obtain a commutative diagram of direct image homomorphisms:
\[
\vcenter{\xymatrix{
& X\otimes_{I(X)} X \ar[dr]^{(\pi_1)_!} \\
 Q\otimes_{Q_0} X \ar[ur]_{\cong}^{\langle \act, \pi_2 \rangle_!}  \ar[d]_{(\pi_1)_!} \ar[rr]^-{\act_!} && X \ar[d]^{p_!}  \\
 G_1  \ar[rr]_-{d_!} && Q_0 }
}
\]
Furthermore, since $[a^*,\pi^*_2]$ is the inverse of $\langle \act, \pi_2 \rangle_!$, we conclude that
\[  d_!\circ (\pi_1)_!\circ [a^*,\pi^*_2] = p_!\circ (\pi_1)_!\;,   \]
so given $s,t\in\bisections_X$ such that $\tspp(s)=\tspp(t)$ we have $\spp_X(s)=\spp_Q(\langle s,t\rangle)$, as the following derivation shows:
\begin{align*}
\spp_X(s) &= \spp_X(s\tspp(s))\\
&= \spp_X(s\tspp(t))& \text{(because $\tspp(s)=\tspp(t)$)}\\
&= \spp_X((\pi_1)_!(s\otimes t))& \text{(by Lemma~\ref{directimage})}\\
&= p_!((\pi_1)_!(s\otimes t))\\
&=d_!\circ (\pi_1)_!\circ [a^*,\pi^*_2](s\otimes t)\\
&=d_!\circ (\pi_1)_!(\langle s,t\rangle\otimes t)\\
&=d_!(\langle s,t\rangle\spp_X(t))=d_!(\langle s,\spp_X(t)t\rangle)\\
&=\spp_Q(\langle s,t\rangle)\;.
\end{align*}  
Similarly we can prove $\spp_Q(\langle s,t\rangle^*)=\spp_X(t)$. Finally, due to Lemma~\ref{tut} we have $\tspp(\langle t,s\rangle s)=\tspp(s)=\tspp(t)$, and thus we obtain  
\[
\langle t,s\rangle s\le t\iff\langle t,s\rangle s = t\tspp(\langle t,s\rangle s)=t\tspp(t)=t\;,
\]
which implies that $ut=s$, as intended, where $u=\langle s,t \rangle\in Q_\ipi$. The uniqueness of $u$ is due to the principality of the local bisections, see Lemma~\ref{pairsofprincipal}. 
\end{proof}

\begin{corollary}\label{funct2, cor: principal-properties}
Let $Q$ be an inverse quantal frame and $X$ a principal $Q$-sheaf. The following property holds for all $s,s',t\in \bisections_X$ such that $\tspp(s)=\tspp(s')=\tspp(t)$:
\begin{equation}\label{eq: innermeet1}
{\langle s\wedge s',t \rangle}={\langle s,t \rangle}\wedge {\langle s',t \rangle}\;.
\end{equation}
\end{corollary}

\begin{proof}
By Theorem~\ref{funct2, theo: principalsheaves2}, we can write $u={\langle s,t \rangle}\in Q_\ipi$ (where $\spp_Q(u)=\spp_X(s)$, $\spp_Q(u^*)=\spp_X(t)$ and $ut=s$) and $v={\langle s',t \rangle}\in Q_\ipi$ (where $\spp_Q(u)=\spp_X(s')$, $\spp_Q(u^*)=\spp_X(t)$ and $vt=s'$). Our goal is to show that $u\wedge v={\langle s\wedge s',t\rangle} $. Clearly, ${\langle s\wedge s',t\rangle}\leq u\wedge v$. To show the converse, let us consider:
\begin{equation}\label{innermeet}
{\langle s\wedge s',t\rangle} = \V\{ b\in Q_\ipi\st \spp_Q(b)\leq \spp_X(s\wedge s'),\, \spp_Q(b^*)\leq \spp_X(t),\, bt\leq s\wedge s'     \}\;,
\end{equation}
and let us prove that $u\wedge v$ satisfies the conditions of \eqref{innermeet}, as follows:
\begin{itemize}
\item $\spp_Q((u\wedge v)^*)\leq \spp_X(t)$ --- because $\spp_Q((u\wedge v)^*)\leq \spp_Q(u^*)=\spp_X(t)$.
\item $(u\wedge v)t\leq s\wedge s'$ --- because $(u\wedge v)t\leq ut\wedge vt=s\wedge s'$. 
\item $\spp_Q(u\wedge v)\leq \spp_X(s\wedge s')$ --- Indeed, by applying the stability of $\spp_Q$ we have
\begin{align*}
\spp_Q(u\wedge v)&=(u\wedge v)(u\wedge v)^*\\
&= \spp_Q((u\wedge v)(u\wedge v)^*)\\
&= \spp_Q((u\wedge v)\spp_Q((u\wedge v)^*))\\
&\leq \spp_Q((u\wedge v)\spp_X(t))\\
&= \spp_X((u\wedge v)t)\\
&\leq \spp_X(s\wedge s')\;.
\end{align*}
\end{itemize}
Hence, $u\wedge v \leq \langle s\wedge s',t\rangle$ and the equality holds. 
\end{proof}

\subsection{Principal bibundles revisited}

Let $Q$ and $R$ be inverse quantal frames. By a \emph{$Q$-$R$-bisheaf} will be meant a $Q$-$R$-bimodule $X$ that is both a $Q$-sheaf and a (right) $R$-sheaf.
Equivalently, a $Q$-$R$-bisheaf is a $Q$-$R$-bimodule $X$ that satisfies the following conditions (\cf\ Lemma~\ref{lem:bisheafbasis}):
\begin{description}
\item[Inner products:] $X$ is equipped with two inner products
\[
\langle-,-\rangle: X\times X\to Q\quad\quad\text{and}\quad\quad [-,-]:X\times X\to R
\]
(with $\langle-,-\rangle$ being left $Q$-linear and $[-,-]$ being right $R$-linear);
\item[Bisections:] There is $\sections\subset X$ such that for all $x\in X$
\[
\V_{s\in\sections} \langle x,s\rangle s = \V_{s\in\sections} s[s,x]=x\;.
\]
\end{description}
The largest such set $\sections$ will be denoted by $\bisections_X$, and its elements are called \emph{local bisections} of $X$ (\cf\ section~\ref{sec:transitivity}).
We shall usually denote the corresponding support operators by
\[
\spp_X :X\to Q_0\quad\quad\text{and}\quad\quad\tspp :X\to R_0\;,
\]
so for all $x\in X$ we have
\[
\spp_X(x) = \langle x,x\rangle\wedge e_Q\quad\quad\text{and}\quad\quad\tspp(x)=[x,x]\wedge e_R\;.
\]

Analogously, for \'etale groupoids $G$ and $H$, a $G$-$H$-bisheaf is a $G$-$H$-bilocale $X$ whose left and right projections $G_0\stackrel p\longleftarrow X\stackrel q\longrightarrow H_0$ are local homeomorphisms. Clearly, $X$ is a $G$-$H$-bisheaf if and only if it is an $\opens(G)$-$\opens(H)$-bisheaf, in which case we have, writing $u$ for the unit inclusion maps of both $G$ and $H$,
\[
\spp_X=u_!\circ p_!\quad\quad\text{and}\quad\quad \tspp=u_!\circ q_!\;.
\]
In addition, by a \emph{principal $G$-$H$-bisheaf} will be meant a principal $G$-$H$-bibundle which is also a $G$-sheaf (necessarily a $G$-$H$-bisheaf, by Lemma~\ref{funct2, lemma: covering}).

\begin{theorem}\label{GHbisheaves}
Let $G$ and $H$ be \'etale groupoids with quantales $Q:=\opens(G)$ and $R:=\opens(H)$, and let $X$ be a $G$-$H$-bisheaf (equivalently, a $Q$-$R$-bisheaf). Then $X$ is a principal $G$-$H$-bisheaf if and only if the following conditions hold:
\begin{enumerate}
\item\label{GHbisheaves1} $\langle s,s\rangle\le e_Q$ for all $s\in\bisections_X$;
\item\label{GHbisheaves2} $[1_X,1_X]\ge e_R$;
\item\label{GHbisheaves3} $1_X\tspp(s)\le 1_Q s$ for all $s\in\bisections_X$.
\end{enumerate}
Moreover, the left projection $p:X\to G_0$ is a surjection if and only if
\begin{enumerate}
\setcounter{enumi}{3}
\item\label{GHbisheaves4} $\langle 1_X,1_X\rangle\ge e_Q$,
\end{enumerate}
and in this case $\langle-,-\rangle:X\times X\to Q$ is surjective.
\end{theorem}

\begin{proof}
Since $p$ is an open map, being a surjection is equivalent to the condition $p_!(1_X)=1_{G_0}$, which is equivalent to \eqref{GHbisheaves4} because
$u_!$ is injective and $u_!(1_{G_0})=e_Q$, and $\langle 1_X,1_X\rangle\wedge e_Q=\spp_X(1_X)=u_!(p_!(1_X))$. In this case the surjectivity of $\langle-,-\rangle$ follows from Lemma~\ref{funct2, lemma: b_!=inner}. Now let us assume that \eqref{GHbisheaves1}--\eqref{GHbisheaves3} hold. Then \eqref{GHbisheaves2} implies that $q:X\to H_0$ is a surjection, so $q_!\circ q^*=\ident_{H_0}$. Moreover, $q^*(c)$ equals $1_X u_!(c)$ (where now $u:H_0\to H_1$), so $q^*$ is valued in $I(X)$: $1_Q q^*(c) = 1_Q (1_X u_!(c)) = (1_Q 1_X) u_!(c)=1_X u_!(c)=q^*(c)$. And \eqref{GHbisheaves3} implies $1_X\tspp(1_Qx)\le 1_Q x$ for all $x\in X$, which is equivalent to $q^*(q_!(x))\le x$ for all $x\in I(X)$, and thus to $q^*\circ q_!=\ident_{I(X)}$ because $q^*$ is right adjoint to $q_!$. Hence, \eqref{GHbisheaves2} and \eqref{GHbisheaves3} imply that $1_X(-):R_0\to I(X)$ is an order isomorphism and that the following diagram commutes:
\[
\xymatrix{
&&X\\
R_0\ar[urr]|{1_X(-)}\ar[rrrr]_{1_X(-)}^\cong&&&& I(X)\ar[ull]|{\text{inclusion}}
}
\]
Writing $\pi:X\to I(X)$ for the map of locales defined by the inclusion $\pi^*:I(X)\to X$, we obtain a commutative diagram in $\Loc$:
\[
\xymatrix{
&&X\ar[dll]_q\ar[drr]^\pi\\
H_0\ar[rrrr]^\cong&&&& I(X)
}
\]
Since moreover $X$ is principally covered due to \eqref{GHbisheaves1} it follows that $X$ is a principal $G$-$H$-bisheaf. Showing that being a principal $G$-$H$-bisheaf implies \eqref{GHbisheaves1}--\eqref{GHbisheaves3} is, using Lemma~\ref{McongX/G}, a straightforward reversal of the previous arguments.
\end{proof}

\begin{remark}\label{rem:GHbishesves}
Notice that due to \eqref{eq:sppXsppQinner} some of the conditions in the statement of the above theorem can be replaced as follows:
\begin{itemize}
\item Condition~\eqref{GHbisheaves2} is equivalent to $\V_{s\in\bisections_X}[s,s]\ge e_R$;
\item Condition~\eqref{GHbisheaves4} is equivalent to $\V_{s\in\bisections_X}\langle s,s\rangle\ge e_Q$;
\item In the case that $p:X\to G_0$ is a surjection, \eqref{GHbisheaves1} and \eqref{GHbisheaves4} can be replaced by the single condition $\V_{s\in\bisections_X} \langle s,s\rangle = e_Q$.
\end{itemize}
\end{remark}

Let $Q$ and $R$ be inverse quantal frames. A \emph{Hilsum--Skandalis sheaf} from $R$ to $Q$ is the isomorphism class of a principal $Q$-$R$-bisheaf. The \emph{category $\HSQ$} is that whose objects are the inverse quantal frames and whose morphisms are the Hilsum--Skandalis sheaves, with composition defined in terms of tensor products, as follows: if $Q$, $R$, and $S$ are inverse quantal frames and $\psi:S\to R$ and $\varphi:R\to Q$ are Hilsum--Skandalis sheaves represented by a principal $R$-$S$-bisheaf $Y$ and a principal $Q$-$R$-bisheaf $X$, respectively, then the composition $\varphi\circ\psi$ is represented by the tensor product $X\otimes_R Y$, which is a principal $Q$-$S$-bisheaf.

Defining $\HSGsh$ to be the subcategory of $\HSG$ obtained by restricting to Hilsum--Skandalis maps represented by principal bibundles that are also sheaves, we immediately obtain:

\begin{theorem}
The bi-equivalence of \cite{Re15}*{Corollary 4.12} yields an equivalence between the categories $\HSGsh$ and $\HSQ$.
\end{theorem}

\begin{remark}
Let $Q$ and $R$ be inverse quantal frames. By Lemma~\ref{lem:globalsection}, a Hilsum--Skandalis sheaf $R\to Q$ is the isomorphism class of $\langle\varphi\rangle$ for a groupoid functor $\varphi:H\to G$, where $Q\cong\opens(G)$ and $R\cong\opens(H)$, if and only if there is $s\in\tsections_X$ such that $\tspp(s)=e_R$.
\end{remark}

\subsection{Biprincipal bisheaves}

Let $Q$ and $R$ be inverse quantal frames. A principal $Q$-$R$-bisheaf $X$ will be said to be \emph{biprincipal} if it is also principal as an $R$-sheaf; equivalently, if $X^*$ is a principal $R$-$Q$-bisheaf.
From Theorem~\ref{GHbisheaves} we immediately obtain (\cf\ Remark~\ref{rem:GHbishesves}):

\begin{corollary}\label{cor:GHbisheaves}
Let $Q$ and $R$ be inverse quantal frames and $X$ a $Q$-$R$-bisheaf. Then $X$ is a biprincipal $Q$-$R$-bisheaf if and only if the following conditions hold:
\begin{enumerate}
\item\label{cor:GHbisheaves1} $\V_{s\in\bisections_X} \langle s,s\rangle = e_Q$;
\item\label{cor:GHbisheaves2} $\V_{s\in\bisections_X}[s,s]= e_R$;
\item\label{cor:GHbisheaves3} $1_X\tspp(s)\le 1_Q s$ for all $s\in\bisections_X$;
\item\label{cor:GHbisheaves4} $\spp_X(s)1_X\le s1_R$ for all $s\in\bisections_X$.
\end{enumerate}
Furthermore, if these conditions hold then $\langle-,-\rangle$ and $[-,-]$ are surjective maps: $\langle X,X\rangle=Q$ and $[X,X]=R$.
\end{corollary}

In the above description of biprincipal bisheaves conditions \eqref{cor:GHbisheaves3} and \eqref{cor:GHbisheaves4} can be replaced by the (much easier to remember) \emph{interchange rule} $\langle x,y\rangle z=x[y,z]$ that is common in definitions of equivalence (or imprimitivity) bimodules for C*-algebras and inverse semigroups:

\begin{theorem}\label{funct2, lemma: biprincipal-properties}
Let $Q$ and $R$ be inverse quantal frames and $X$ a $Q$-$R$-bisheaf satisfying
\[
\V_{s\in\bisections_X} \langle s,s\rangle = e_Q\quad\quad\text{and}\quad\quad
\V_{s\in\bisections_X}[s,s]= e_R\;.
\]
Then $X$ is a biprincipal $Q$-$R$-bisheaf if and only if for all $s,t,u\in\bisections_X$ we have
\begin{equation}\label{bipr}
\langle s,t\rangle u= s[t,u]\;.
\end{equation}
\end{theorem}

\begin{proof}
First let us assume that $X$ is biprincipal and prove that \eqref{bipr} holds.
Let $s,t,u\in \bisections_X$. Note that
\[
\langle s,t\rangle u=\langle s,\spp_X(t)t\rangle\spp_X(u)u
=\langle s,\spp_X(u)t\rangle\spp_X(t)u\;,
\]
so we may assume without loss of generality that $\spp_X(t)=\spp_X(u)$ in the expression $\langle s,t\rangle u$. In addition, using Theorem~\ref{axyxay} we obtain
\[
\langle s,t\rangle u=\langle s\tspp(s),t\tspp(t)\rangle u = \langle s\tspp(t),t\tspp(s)\rangle u\;,
\]
so again without loss of generality we shall assume that in the expression $\langle s,t\rangle u$ we have $\tspp(s)=\tspp(t)$. We note that in the expression $s[t,u]$ exactly the same assumptions can be made, for similar reasons: $\spp_X(t)=\spp_X(u)$ and $\tspp(s)=\tspp(t)$.
Then, since $\langle s,t\rangle\in Q_\ipi$,
by Theorem~\ref{funct2, theo: principalsheaves2} we conclude that $\langle s,t\rangle$ is the unique $v\in Q_\ipi$ such that $\spp_Q(v^*)=\spp_X(t)$ and $vt=s$. Similarly, $[ t,u ]$ is the unique $w\in R_\ipi$ such that $\spp_R(w)=\tspp(t)$ and $tw=u$. Therefore
\[ \langle s,t\rangle u=vu=vtw=sw=s[t,u]\;,\]
so \eqref{bipr} holds.
Now let us assume that \eqref{bipr} holds and prove that $X$ is a biprincipal $Q$-$R$-bisheaf. This means that we must prove the conditions of Corollary~\ref{cor:GHbisheaves}\eqref{cor:GHbisheaves3} and Corollary~\ref{GHbisheaves}\eqref{GHbisheaves4}. Let $s\in\bisections_X$. Then
\[
1_Q s\ge \langle 1_X,s\rangle s=1_X[s,s]=1_X\tspp(s)\;,
\]
so Corollary~\ref{cor:GHbisheaves}\eqref{cor:GHbisheaves3} holds. Similarly,
\[
s1_R\ge s[s,1_X]=\langle s,s\rangle 1_X=\spp_X(s) 1_X\;,
\]
so Corollary~\ref{cor:GHbisheaves}\eqref{cor:GHbisheaves4} holds.
\end{proof}

Using the above result another useful description of biprincipal bisheaves is obtained:

\begin{theorem}\label{funct2, theo: ME}
Let $Q$ and $R$ be inverse quantal frames and $X$ a $Q$-$R$-bisheaf. The following conditions are equivalent:
\begin{enumerate}
\item\label{ME2} $X$ is a biprincipal $Q$-$R$-bisheaf;
\item\label{ME3} $X$ is a principal $Q$-$R$-bisheaf with inverse given by $X^*$ in $\HSQ$.
\end{enumerate}
\end{theorem}

\begin{proof}
The implication \eqref{ME3}$\Rightarrow$\eqref{ME2} is immediate, so let us prove \eqref{ME2}$\Rightarrow$\eqref{ME3}. Let $X$ be a principal $Q$-$R$-bisheaf. The inner product $\langle-,-\rangle:X\times X\to Q$ defines a sup-lattice homomorphism
\[
\varphi:X\otimes_R X^*\to Q
\]
by
$\varphi(x\otimes y)=\langle x,y\rangle$
because by Theorem~\ref{axyxay} we have, for all $x,y\in X$ and $r\in R$,
\[
\varphi(xr\otimes y)=\langle xr,y\rangle = \langle x,yr^*\rangle
=\varphi(x\otimes(r\cdot y))\;.
\]
In addition, $\varphi$ is surjective because by Corollary~\ref{cor:GHbisheaves} $\langle-,-\rangle$ is surjective. And it is both left and right $Q$-equivariant, hence being a homomorphism of $Q$-$Q$-bimodules, as the following derivations with $x,y\in X$ and $a\in Q$ show:
\begin{eqnarray*}
\varphi(ax\otimes y) &=& \langle ax,y\rangle = a\langle x,y\rangle =a\varphi(x\otimes y)\;;\\
\varphi(x\otimes(y\cdot a)) &=& \langle x,a^*y\rangle = \langle x,y\rangle a=\varphi(x\otimes y)a\;.
\end{eqnarray*}
Similarly, there is a surjective homomorphism of $R$-$R$-bimodules
\[
\psi:X^*\otimes_Q X\to R
\]
defined by $\psi(x\otimes y)=[x,y]$ for all $x,y\in X$.
Now let us prove that $\varphi$ is injective. In order to do that we shall find a mapping $\eta:Q\to X\otimes_R X^*$ such that $\eta\circ\varphi=\ident$. Let us define $\eta$ as follows for all $a\in Q$: 
\[
\eta(a):=\V_{s\in \bisections_X} as\otimes s\;.
\]
Let $t,u\in \bisections_X$. Then, using \eqref{bipr}, we have
\begin{align*}
\eta(\varphi(t\otimes u)) &= \V_{s\in\bisections_X} \langle t,u\rangle s\otimes s= \V_{s\in\bisections_X} t[u,s]\otimes s\\
&= \V_{s\in\bisections_X} t\otimes ([u,s]\cdot s) = t\otimes\V_{s\in\bisections_X} s[s,u]=t\otimes u\;.
\end{align*}
Then, since any $\xi\in X\otimes_R X^*$ must be of the form $\V_i t_i\otimes u_i$ for some $t_i,u_i\in\bisections_X$, and noticing that $\eta$ is obviously a sup-lattice homomorphism, we obtain $\eta(\varphi(\xi))=\xi$ for all $\xi\in X\otimes_R X^*$, so $\varphi$ is injective. Similarly one shows that $\psi$ is injective.
\end{proof}

\subsection{Morita equivalence}

Let us define two inverse quantal frames to be \emph{Morita equivalent} if their corresponding \'etale groupoids are Morita equivalent. The following facts are immediate and show that biprincipal bisheaves provide a good notion of ``Morita equivalence bimodule'' in the context of inverse quantal frames:

\begin{corollary}\label{mmain}
Let $Q$ and $R$ be inverse quantal frames. The following conditions are equivalent:
\begin{enumerate}
\item\label{mmain1} $Q$ and $R$ are Morita equivalent;
\item\label{mmain2} There exists a principal $Q$-$R$-bisheaf $X$, a principal $R$-$Q$-bisheaf $Y$, and isomorphisms of $Q$-$Q$-bilocales and $R$-$R$-bilocales such that
$X\otimes_Q Y\cong R$ and $Y\otimes_R X\cong Q$;
\item\label{mmain3} There exists a principal $Q$-$R$-bisheaf $X$ such that $X^*$ is a principal $R$-$Q$-bisheaf, and isomorphisms of $Q$-$Q$-bilocales and $R$-$R$-bilocales such that $X\otimes_Q X^*\cong R$ and $X^*\otimes_R X\cong Q$;
\item\label{mmain4} There exists a biprincipal $Q$-$R$-bisheaf.
\item\label{mmain5} The categories $Q$-$\Sh$ and $R$-$\Sh$, respectively of $Q$-sheaves and $R$-sheaves, are equivalent.
\end{enumerate}   
\end{corollary}

\begin{proof}
The equivalence of conditions \eqref{mmain1}--\eqref{mmain3} follows from Lemma~\ref{lem:moritaequivdefs} and Lemma~\ref{funct2, lemma: covering}, and these are equivalent to \eqref{mmain4} due to Lemma~\ref{funct2, theo: ME}. The equivalence to \eqref{mmain5} is a consequence of the fact that two \'etale groupoids $G$ and $H$ are Morita equivalent if and only if their classifying toposes $BG$ and $BH$ are equivalent~\cite{Moer90}, together with the fact that the classifying topos $BG$ of an \'etale groupoid $G$ is equivalent to $\opens(G)$-$\Sh$~\cite{GSQS}.
\end{proof}

These results yield yet another equivalent definition of Morita equivalence of \'etale groupoids (adding to those of Lemma~\ref{lem:moritaequivdefs}), agreeing with the fact that the existence of a biprincipal bibundle (or equivalent variants of this) is often taken as a definition of Morita equivalence for topological groupoids~\cites{Meyer97,MRW87}:

\begin{corollary}
Let $G$ and $H$ be \'etale groupoids. The following conditions are equivalent:
\begin{enumerate}
\item $G$ and $H$ are Morita equivalent;
\item There exists a \emph{biprincipal $G$-$H$-bibundle} (a principal $G$-$H$-bibundle which is also principal as an $H$-bundle).
\end{enumerate}
\end{corollary}

Biprincipal bisheaves provide a bimodule based definition of Morita equivalence for structures to which inverse quantal frames are naturally associated. One incarnation of this idea is, of course, that two \'etale groupoids $G$ and $H$ are Morita equivalent if and only if there is a biprincipal $\opens(G)$-$\opens(H)$-bisheaf. In the same spirit, our results may have applications to pseudogroups via the equivalence of categories that to each pseudogroup $S$ associates its inverse quantal frame $\lcc(S)$~\cite{Re07}. The Morita theory of inverse semigroups has been extensively studied~\cite{FLS11}, but pseudogroups carry more topological information than general inverse semigroups, in particular because their idempotents form locales, and this should be taken in account. It is natural to define two pseudogroups $S$ and $T$ to be Morita equivalent if the inverse quantal frames $\lcc(S)$ and $\lcc(T)$ are. We shall not pursue this in this paper, except to note that some simplifications of presentation arise from the fact that the inner products of local bisections of a biprincipal $\lcc(S)$-$\lcc(T)$-bisheaf are partial units of the quantales, which means that we can take the inner products to be valued in the pseudogroups themselves, leading to a definition that is closely related to the equivalence bimodules of~\cite{Steinberg-morita}.

\vspace*{5mm}
\noindent {\sc
Centro de An\'alise Matem\'atica, Geometria e Sistemas Din\^amicos
Departamento de Matem\'{a}tica, Instituto Superior T\'{e}cnico\\
Universidade T\'{e}cnica de Lisboa\\
Av.\ Rovisco Pais 1, 1049-001 Lisboa, Portugal}\\
{\it E-mail:} {\sf quijano.juanpablo@gmail.com}, {\sf pmr@math.tecnico.ulisboa.pt}

\end{document}